\documentclass[preprint,12pt]{elsarticle}
\usepackage{a4wide,amsmath,amssymb,amsthm,graphicx,color,comment}
\newtheorem{theorem}{Theorem}[section]

\newtheorem{lemma}[theorem]{Lemma}
\newtheorem{remark}[theorem]{Remark}
\newcommand{\uld}[1]{\underline{d#1}}
\newcommand\rtil{\tilde{r}}
\newcommand\radius{\mathfrak{r}}
\newcommand\taubar{\overline{\tau}}
\definecolor{darkcyan}{rgb}{0.,0.5,0.5}
\definecolor{darkgreen}{rgb}{0.,0.7,0.}
\newcommand{\commBK}[1]{}
\newcommand{\skipKuznetsov}[1]{}
\journal{Journal of Differential Equations}

\begin{document}

\begin{frontmatter}

\title{Well-posedness of the time-periodic Jordan-Moore-Gibson-Thompson equation}
\author{Barbara Kaltenbacher, University of Klagenfurt}
\address{${}^1$Alpen-Adria-Universit\"at Klagenfurt (barbara.kaltenbacher@aau.at)
}

\begin{abstract}
Motivated by applications of nonlinear ultrasonics under continuous wave excitation, we study the Jordan-Moore-Gibson-Thompson (JMGT) equation -- a third order in time 
quasilinear PDE -- under time periodicity conditions. 
Here the coefficient of the third order time derivative is the so-called relaxation time and a thorough understanding of the limiting behaviour for vanishing relaxation time is essential to link these JMGT equations to classical second order models in nonlinear acoustics,
 
As compared to the meanwhile well understood initial value problem for JMGT, the periodic setting poses substantial challenges due to a loss of temporal regularity, while the analysis still requires an $L^\infty$ control of solutions in space and time in order to maintain stability or equivalently, to avoid degeneracy of the second time derivative coefficient.

We provide a full well-posedness analysis with and without gradient nonlinearity, as relevant for modelling 
non-cumulative
nonlinear effects, under practically relevant mixed boundary conditions. 
The source-to-state map is thus well-defined and we additionally show it to be Lipschitz continuously differentiable, a result that is useful for inverse problems applications such as acoustic nonlinearity tomography. The energy bounds derived for the well-posedness analysis of periodic JMGT equations also allow to fully justify the singular limit for vanishing relaxation time.
\end{abstract}



\end{frontmatter}

\section{Introduction}
Motivated by applications, in particular of high-intensity ultrasound~\cite{abramov,kennedy2003high, wu2001pathological, yoshizawa2009high} and nonlinear ultrasound tomography~\cite{nonlinparam1, duck2002nonlinear,  nonlinparam3, nonlinparam2, ZHANG20011359,nonlinearity_imaging_JMGT,nonlinearity_imaging_Westervelt}, the field of modeling and analysis of nonlinear acoustics has recently found much interest.

In this paper, we analyze a class of nonlinear third order in time acoustic wave equations that has been put forward in~\cite{JordanMaxwellCattaneo09,JordanMaxwellCattaneo14} and initial-boundary value problems for these PDEs have been intensively studied, along with their linearizations, in the recent literature; we refer to, 
e.g.,~\cite{BongartiCharoenphonLasiecka20,bucci2020regularity,chen2019nonexistence,DellOroPata,KLM12_MooreGibson,KLP12_JordanMooreGibson,JMGT,JMGT_Neumann,LiuTriggiani13,MarchandMcDevittTriggiani12,PellicerSolaMorales,racke2020global} for some selected results on well-posedness, regularity of solutions, and long-term behavior of initial value problems for these equations.  Also the influence of memory terms on the solution behaviour has been extensively investigated; see, e.g.,~\cite{dell2016moore, dell2020note,LasieckaWang15b, LasieckaWang15a,  lasiecka2017global, LasieckaWang15a} and the references provided therein.
Adhering to the meanwhile commonly used denomination as Jordan-Moore-Gibson-Thompson JMGT type equations (with ``J'' being skipped in the linear case), and distinguishing between two types of nonlinearity inherited from classical models (see also \eqref{Westervelt}, \eqref{Kuznetsov} below) we will study 
\begin{itemize}
\item 
the JMGT-Westervelt equation for the acoustic pressure $u$
\begin{equation}\label{JMGT-Westervelt}
\tau u_{ttt}+u_{tt}-c^2\Delta u -b \Delta u_t + \eta (u^2)_{tt} +f=0,
\end{equation}
\item 
the JMGT-Kuznetsov equation for the acoustic velocity potential $u$
\begin{equation}\label{JMGT-Kuznetsov}
\tau u_{ttt}+u_{tt}-c^2\Delta u -b \Delta u_t +\bigl(\tilde{\eta}u_t^2+|\nabla u|^2\bigr)_t +f=0,
\end{equation}
\end{itemize}
on some space-time domain $(0,T)\times\Omega$ with $\Omega\in\mathbb{R}^d$, $d\in\{1,2,3\}$.
Here $\tau>0$ is the relaxation time, $c>0$ the speed of sound and $b>0$ an attenuation coefficient that is related to the diffusivity of sound.
We will consider both versions in a unified way  
\begin{equation}\label{JMGT-linearized}
\tau u_{ttt}+\alpha u_{tt}-c^2\Delta u -b \Delta u_t +\rtil=0,
\end{equation}
with the two cases
\begin{equation}\label{alphar}
\begin{aligned}
&
\alpha=1+2\eta u, \ \rtil=2\eta (u_t)^2+f 
\\
&
\alpha=1+2\tilde{\eta} u_t, \ \rtil=2\nabla u\cdot\nabla u_t+f 
.
\end{aligned}
\end{equation}
With general given $\alpha$ and $\rtil$, \eqref{JMGT-linearized} also comprises a linear MGT equation, that will be used in the analysis.

Motivated by applications with waves exhibiting temporal periodicity (e.g., under sinosoidal continuous wave excitations), we consider the 
conditions
\begin{equation}\label{periodic}
u(T)=u(0), \quad u_t(T)=u_t(0), \quad \tau u_{tt}(T)=\tau u_{tt}(0) 
\end{equation}
for some $T>0$; in particular, if $\tau=0$ and thus the equations above are second order in time, the periodicity condition on $u_{tt}$ is omitted.
An analysis in this time periodic setting is so far missing, to the best of the author's knowledge, but is highly desirable due to its practical relevance.

For the most basic model of nonlinear acoustics, the Westervelt equation
\begin{equation}\label{Westervelt}
u_{tt}-c^2\Delta u -b \Delta u_t + \eta (u^2)_{tt} +f=0
\end{equation}
an analysis under periodicity conditions \eqref{periodic} has been carried out in \cite{periodicWestervelt, periodicWest_2}.
Taking into account 
non-cumulative
nonlinear effects leads to the more advanced Kuznetsov equation
\begin{equation}\label{Kuznetsov}
u_{tt}-c^2\Delta u -b \Delta u_t +\bigl(\tilde{\eta}u_t^2+|\nabla u|^2\bigr)_t +f=0.
\end{equation}
Both can obviously be obtained by formally setting $\tau=0$ in \eqref{JMGT-Westervelt}, \eqref{JMGT-Kuznetsov}.
One of the tasks here will be to justify the singular limit $\tau\searrow0$ in appropriate topologies.

On the boundary $\partial\Omega=\Gamma_a\cup\Gamma_i\cup\Gamma_N\cup\Gamma_D$ of the spatial domain where \eqref{JMGT-Westervelt} or \eqref{JMGT-Kuznetsov} or \eqref{JMGT-linearized} is supposed to hold, in order to capture a wide range of practically relevant scenarios, we impose  
mixed boundary conditions 
\begin{equation}\label{bndy}
\begin{aligned}
&\partial_\nu u+\beta u_t+\gamma u =0 \mbox{ on }\Gamma_a\cup\Gamma_i\cup\Gamma_N\\
&u =0 \mbox{ on }\Gamma_D\\
&\text{ where $\gamma=0$ on $\Gamma_N$, $\tfrac{1}{\gamma}\vert_{\Gamma_i}\in L^\infty(\Gamma_i)$, $\beta=0$ on $\Gamma_i\cup\Gamma_N$,}
\end{aligned}
\end{equation}
and consider homogeneous boundary data for simplicity \footnote{pointing to, e.g.,  \cite{FSI} for a standard homogenization approach, applied to a nonlinear wave equation to deal with inhomogeneous boundary conditions}.
This corresponds to 
\underline{a}bsorbing boundary conditions on $\Gamma_a$ to avoid spurious reflections and mimick open domain wave propagation; as well as 
\underline{i}mpedance, \underline{N}eumann, and \underline{D}irichlet boundary conditions on $\Gamma_i$, $\Gamma_N$, and $\Gamma_D$, to model damping, sound-hard and sound-soft boundary parts, respectively. 
We will assume that
\begin{equation}\label{GammaiGammaD}
\text{meas}(\Gamma_i)+\text{meas}(\Gamma_D)>0,
\end{equation}
and in norms frequently use the abbreviation  
\begin{equation}\label{Gamma}
\Gamma=\Gamma_a\cup\Gamma_i. 
\end{equation}

An essential difficulty in the analysis of the PDEs \eqref{JMGT-Westervelt}, \eqref{JMGT-Kuznetsov}, \eqref{Westervelt}, \eqref{Kuznetsov} arise due to the potential degeneracy of the second time derivative coefficient $\alpha$, which depends on the state $u$, cf. \eqref{alphar}.
For the second order models \eqref{Westervelt}, \eqref{Kuznetsov}, it is immediately clear that they lose their wave type character as soon as $\alpha$ fails to be positive (and bounded away from zero). 
But also for the third order in time equations \eqref{JMGT-Westervelt}, \eqref{JMGT-Kuznetsov}, nondegeneracy of $\alpha$ is critical. In fact, as known from previous work on the (J)MGT equation, a precondition for stability of solutions is strict positivity of the coefficient $\frac{b}{c^2}-\frac{\tau}{\alpha}$.
Correspondingly, with $\tau\in[0,\overline{\tau}]$, we will require
\begin{equation}\label{stabilitycond}
\frac{b}{c^2}-\frac{\overline{\tau}}{\alpha} >0  \quad \text{ and bounded away from zero.}
\end{equation}
which is a condition on the interplay between attenuation $b$, sound speed $c$, relaxation time $\tau$ and second time dervative coefficient $\alpha$.
\footnote{In fact, the damping coeffcient $b$ usually takes the form $b=\delta+\tau c^2$ with $\delta>0$ being the diffusivity of sound. With this, \eqref{stabilitycond} reads as $\frac{\delta}{c^2}> \frac{\overline{\tau}}{\alpha}-\tau$} 
The latter thus needs to be bounded away from zero, which in view of \eqref{alphar}, exhibiting $\alpha$ as a perturbation of unity, necessitates the control of $\eta u$ or $\tilde{\eta} u_t$ in $L^\infty(0,T;L^\infty(\Omega))$. 
In the existing literature referred to above, this is acheved by deriving energy bounds in sufficiently strong Sobolev spaces and using their continuous embedding into $L^\infty(0,T;L^\infty(\Omega))$.
This is inhibited here in two ways: By a temporal loss of regularity due to the periodicity (rather than initial) conditions and by a spatial regularity loss due to mixed boundary conditions.

The key element of our proofs is therefore to derive new energy estimates that cope with both. These are enabled by a broad testing strategy along with a dedicated Galerkin discretization to guarantee feasibility of this testing.

\medskip

The goals of this work are to
\begin{itemize}
\item analyze periodic JMGT type equations for the first time;\\ 
Challenges due to periodicity conditions as compared to the initial value problems studied in the literature so far arise due to the lack of $L^\infty$ in time estimates directly from energy identities, cf. Remark~\ref{rem:compare_init} 
\commBK{Note that terms of the form $\int_0^T \frac{d}{dt} \cdots dt$ just vanish and Gronwall's inequality is not available};
\item justify limits as the relaxation time $\tau$ tends to zero;
\item prove differentiability of the source-to-state map $\mathcal{S}:f\mapsto u$;\\
This will be useful for studying inverse problems related to \eqref{JMGT-Westervelt}, \eqref{JMGT-Kuznetsov} such as nonlinear ultrasound tomography;
\item allow for spatially varying coefficients (with minimal regularity):\\ 
This is motivated by imaging applications, where coefficients may exhibit jumps across tissue or material interfaces. 
\item allow (as much as possible) mixed boundary conditions to achieve a realistic setting; 
Since this impairs regularity of the solutions, we resort to Stampacchia's / De Giorgi's technique. 
(cf., e.g., \cite[Proposition 4.1]{Consiglieri:2014},
\cite[Theorem 4.5 and Section 7.2.1]{Troeltzsch2010}) to control stability via $L^\infty$ bounds in space.
\item An additional novelty in the analysis methodology will be the use of a periodic space-time Galerkin discretization to prove existence of solutions.
\end{itemize}

The structure of the paper is as follows:\\
Section~\ref{sec:mainresults} lists the main results, which are 
\begin{itemize}
\item linear well-posendess on several smoothness levels;
\item nonlinear well-posendess for \eqref{JMGT-Westervelt}, \eqref{JMGT-Kuznetsov}, 
\skipKuznetsov{and \eqref{Kuznetsov},} 
as well as Lipschitz continuous Fr\'{e}chet differentiability of the source-to-state map $\mathcal{S}:f\mapsto u$;
\item singular limits as $\tau\to0$.
\end{itemize}
The proofs are contained in Sections~\ref{sec:linearwellposedness} and \ref{sec:nonlinearwellposedness}.

\subsection*{Notation and some useful facts}
\paragraph{Function spaces:}
For the use of H\"older type estimates in Lebesgue spaces $L^p(\Omega)$, we will use the 
dual index $p^*:=\frac{p}{p-1}$ for $p\in[1,\infty]$ as well as the index $\hat{p}:=2(\tfrac{p}{2})^*=\frac{2p}{p-2}$ for $p\in[2,\infty]$.
The usual Sobolev (Hilbert) spaces are denoted by $H^s(\Omega)$, $s\in(0,\infty)$; some more specifically defined versions are
\begin{equation}\label{H1D}
\begin{aligned}
&H^1_D(\Omega)=\{\phi\in H^1(\Omega)\, : \, \phi=0\text{ on }\Gamma_D\}\\
&H^2_\Delta(\Omega)=\{\phi\in H^1(\Omega)\, : \, \Delta \phi\in L^2(\Omega)\}\\
&H^3_\Delta(\Omega)=\{\phi\in H^1(\Omega)\, : \, \nabla\Delta \phi\in L^2(\Omega)\}.
\end{aligned}
\end{equation}
We will frequently use Bochner spaces and make use of the abbreviations
$L^p(X):=L^p(0,T;X)$, $H^s(X):=H^s(0,T;X)$.\\
By $\text{tr}_{\Gamma_j}$, $j\in\{a,i,N,D\}$ we denote the trace operator on the respective boundary part; we will often skip it when it is clear that we are referring to boundary values of $u$, e.g., under boundary norms.

\paragraph{Constants}
We will write $C_{X\to Y}^\Omega$ for the norm of the embedding $X(\Omega)\to Y(\Omega)$. 

Generally, constants will be denoted by $C$ with certain sub- and/or superscripts indicating their meaning and dependencies. Since we will track dependency on the small constants $\tau$, $\taubar$ $\in[0,1]$ explicitely, all constants denoted by $C$ will be independent of $\tau$ and $\taubar$.

Sometimes we will just write ``$a\lesssim b$'' for ``$a\leq C b$ for some $C>0$ independent of $a,b$''.

\paragraph{Inequalities}
Poincar\'{e} Friedrichs type estimates under condition \eqref{GammaiGammaD}:
\begin{equation}\label{PF}
\|v\|_{H^\ell(\Omega)}^2\leq C_{\ell\,PF}\Bigl(\|\nabla v\|_{L^2(\Omega)}^2+\|v\|_{L^2(\Gamma_i)}^2\Bigr), \quad  \ell\in\{0,1\}, \ v\in H^1(\Omega)
\end{equation}
Elliptic regularity:
\begin{equation}\label{Cells}
\begin{aligned}
&\|v\|_{H^{1+s}(\Omega)}^2\leq C_{ell,s}\Bigl(\|v\|_{L^2(\Omega)}^2+\|\Delta v\|_{L^2(\Omega)}^2+\|\partial_\nu v\|_{H^{s-1/2}(\partial\Omega)}^2\Bigr), \\ 
&\qquad\qquad\qquad v\in L^2(\Omega), \ \Delta v\in L^2(\Omega), \ \partial_\nu v\in H^{s-1/2}(\partial\Omega) 
, \ s\in(0,1]
\end{aligned}
\end{equation}
\begin{equation}\label{Cell2}
\begin{aligned}
&\|v\|_{H^3(\Omega)}^2\leq C_{ell,2}\Bigl(\|v\|_{L^2(\Omega)}^2+\|\Delta v\|_{H^1(\Omega)}^2+\|\partial_\nu v\|_{H^{3/2}(\partial\Omega)}^2\Bigr), \\ 
&\qquad\qquad\qquad v\in H^1(\Omega), \ \Delta v\in H^1(\Omega), \ \partial_\nu v\in H^{3/2}(\partial\Omega), 
\end{aligned}
\end{equation}
see, e.g., \cite{Grisvard},
where in its turn $\|v\|_{L^2(\Omega)}^2$ can be estimated by means of \eqref{PF};\\
In case $\partial\Omega=\Gamma_D\cup\Gamma_i$, this also implies the coercivity estimates
\begin{equation}\label{Cell0}
\begin{aligned}
&\|v\|_{H^1(\Omega)}^2\leq C_{ell,0}\|\Delta v\|_{H^1(\Omega)}^2, \\  
&\qquad\qquad\qquad v\in L^2(\Omega), \ \Delta v\in L^2(\Omega), \ v=0\text{ on }\Gamma_D, \ \partial_\nu v+\gamma v=0\text{ on }\partial\Omega\setminus\Gamma_D. 
\end{aligned}
\end{equation}
with $C_{ell,0}:= M^2\, C_{0\,PF}\, C_{1\,PF}$ and $M=\max\{1,\|\tfrac{1}{\gamma}\vert_{\Gamma_i}\|_{L^\infty(\Gamma_i)}\}$, which follows from 
\[
\begin{aligned}
&\frac{1}{M}\Bigl(\|\nabla v\|_{L^2(\Omega)}^2+\|v\|_{L^2(\Gamma_i)}^2\Bigr)
\leq \int_\Omega |\nabla v|^2\, dx + \int_{\Gamma_i} \gamma v^2\, dS\\
&=\int_\Omega(-\Delta v)\, v\, dx
\leq \frac{M\, C_{0\,PF}}{2}\|\Delta v\|_{L^2(\Omega)}^2+\frac{1}{2M\, C_{0\,PF}}\|v\|_{L^2(\Omega)}^2\\
&\leq \frac{M\, C_{0\,PF}}{2}\|\Delta v\|_{L^2(\Omega)}^2+\frac{1}{2M\, C_{0\,PF}}
C_{0\,PF}\Bigl(\|\nabla v\|_{L^2(\Omega)}^2+\|v\|_{L^2(\Gamma_i)}^2\Bigr).
\end{aligned}
\]

\section{Main results}\label{sec:mainresults}

We first of all provide a well-posedness result for the linearized problem \eqref{JMGT-linearized} on three levels of regularity, ``low'', ``medium'' and ``high'', expressed by means of the energy functionals  
\begin{equation}\label{energies_mainresults}
\begin{aligned}
\mathcal{E}_{lo}(u)&=
\taubar\tau^2\|u_{ttt}\|_{L^2(H^1(\Omega)^*)}^2+
\taubar\|u_{tt}\|_{L^2(L^2(\Omega))}^2 + \|u\|_{H^1(H^1(\Omega))}^2\\
&\qquad+ \taubar \|u_{tt}\|_{L^2(L^2(\Gamma_a))}^2+\|\sqrt{\gamma} u\|_{H^1(L^2(\Gamma))}^2\\
\\[1ex]
\mathcal{E}_{me}(u)&=
\taubar\tau^2\|u_{ttt}\|_{L^2(L^2(\Omega))}^2+
\taubar\|u_{tt}\|_{L^2(H^1(\Omega))}^2 + \|\Delta u\|_{H^1(L^2(\Omega))}^2\\
&\qquad+ \taubar\tau \|u_{ttt}\|_{L^2(L^2(\Gamma_a))}^2+ \|\sqrt{\gamma} u\|_{H^2(L^2(\Gamma))}^2\\
\\[1ex]
\mathcal{E}_{hi}(u)&=
\taubar\tau^2\|u_{ttt}\|_{L^2(H^1(\Omega)^*)}^2+
\taubar\|\Delta u_{tt}\|_{L^2(L^2(\Omega))}^2 + \|\nabla\Delta u\|_{H^1(L^2(\Omega))}^2\\
&\qquad
+ \taubar \|\Delta u_{tt}\|_{L^2(L^2(\Gamma_a))}^2 + \|\gamma \Delta u\|_{H^1(L^2(\Gamma))}^2.
\end{aligned}
\end{equation}
We will also use 
\[
\bar{\mathcal{E}}_{me}(u)=\mathcal{E}_{me}(u)+\mathcal{E}_{lo}(u), \qquad
\bar{\mathcal{E}}_{hi}(u)=\mathcal{E}_{hi}(u)+\mathcal{E}_{me}(u)+\mathcal{E}_{lo}(u).
\]

The following assumptions will be made on the coefficients:
\begin{equation}\label{ass_coeff}
\begin{aligned}
&
b, \tfrac{1}{b}, c, \tfrac{1}{c}\in L^\infty(\Omega;\mathbb{R}^+), \quad 
\alpha, \tfrac{1}{\alpha}\in L^\infty(0,T;L^\infty(\Omega;\mathbb{R}^+)), \ \alpha(T)=\alpha(0)\\
&\beta,\gamma\in L^\infty(\Gamma_a\cup\Gamma_i\cup\Gamma_N;\mathbb{R}^+), \quad
\tfrac{1}{\gamma}\vert_{\Gamma_i}\in L^\infty(\Gamma_i;\mathbb{R}^+), \quad
\tfrac{1}{\beta}\vert_{\Gamma_a}\in L^\infty(\Gamma_a;\mathbb{R}^+).
\end{aligned}
\end{equation}
Additional smoothness will be needed on $b$, $c$, $\alpha$, $\rtil$, depending on the level of solution regularity:
\begin{itemize}
\item for the low regularity setting:
\begin{equation}\label{smallnesscoeffs_lo}
\begin{aligned}
&\max\{\|\alpha_t\|_{L^\infty(L^{q^*}(\Omega)},\,
\|\nabla b\|_{L^\infty(L^2(\Omega))}, \,\|\nabla c^2\|_{L^2(L^2(\Omega))},\, 
\\&\hspace*{4cm}
\|\Delta b\|_{L^\infty(L^{\hat{p}}(\Omega))},\,\|\Delta c^2\|_{L^2(L^{\hat{p}}(\Omega))}\}\leq\radius_{lo}\\
&\text{ and } \partial_\nu b=0, \ \partial_\nu c^2=0\text{ on }\partial\Omega, 
\end{aligned}
\end{equation}
\item for the medium regularity setting:
\begin{equation}\label{smallnesscoeffs_me}
\|\nabla\alpha\|_{L^\infty(L^{\hat{p}}(\Omega))} \leq \radius_{me}
\end{equation}
\item for the high regularity setting:
\begin{equation}\label{smallnesscoeffs_hi}
\max\{\|\Delta\alpha\|_{L^\infty(L^2(\Omega))},\,
\|\nabla\alpha\|_{L^\infty(L^{\hat{p}}(\Omega))},\,
\|\nabla b\|_{L^\infty(\Omega)},\,
\|\nabla c^2\|_{L^\infty(\Omega)}\} \leq \radius_{hi}
\end{equation}
\end{itemize}
with
\begin{equation}\label{qstarphat}
\min\{q^*,\tfrac{\hat{p}}{2}\}\begin{cases}\geq 1\text{ for }d=1\\> 1\text{ for }d=2 \\ \geq\frac{d}{2}\text{ for }d\geq 3\end{cases}
\end{equation}
for some $\radius_{lo}, \ \radius_{me}, \ \radius_{hi}>0$.
Note that with $\hat{p}=2(\tfrac{p}{2})^*$, the constraint on $\hat{p}$ in \eqref{qstarphat} implies that $H^1(\Omega)$ continuously embeds into $L^p(\Omega)$.

The proof of the following theorem is carried out in Section~\ref{sec:linearwellposedness}.
\begin{theorem}\label{thm:JMGT-linearized}
Let $\Omega$ be a Lipschitz domain and $T>0$; assume that $\taubar$, $b$, $c$, $\alpha$ satisfy  \eqref{stabilitycond} and \eqref{ass_coeff}.
\begin{enumerate}
\item[(i)] There exist $\radius_{lo}>0$, $C_{lo}>0$ such that for any $\tau\in[0,\taubar]$, and any $b$, $c$, $\alpha$ satisfying \eqref{smallnesscoeffs_lo}
and for any 
\begin{equation}\label{Zlo}
\rtil\in Z_{lo}:=L^2(0,T;L^2(\Omega)),
\end{equation}
the time periodic system
\eqref{JMGT-linearized}, 
\eqref{periodic},
\eqref{bndy}
has a unique solution  
\begin{equation}\label{Ulo}
u\in U_{lo}:=H^3(0,T;H^1(\Omega)^*)\cap H^2(0,T;L^2(\Omega))\cap H^1(0,T;H^1(\Omega))
\end{equation}
and 
\begin{equation}\label{enest_lo}
\mathcal{E}_{lo}(u)\leq C_{lo}( \taubar\|\rtil\|_{L^2(L^2(\Omega))}^2+\|\rtil\|_{L^2(H^1(\Omega)^*)}^2).
\end{equation} 
\item[(ii)] There exist 
$\radius_{lo}>0$, 
$\radius_{me}>0$, $C_{me}>0$ such that 
for any $\tau\in[0,\taubar]$, and any $b$, $c$, $\alpha$ satisfying \eqref{smallnesscoeffs_lo},
\eqref{smallnesscoeffs_me},
and for any 
\begin{equation}\label{Zme}
\begin{aligned}
\rtil\in Z_{me}&:=\{\rtil=\rtil^\nabla+\rtil^t\in L^2(0,T;L^2(\Omega))\ : \\  
&\hspace*{1cm}\rtil^\nabla\in H^1(0,T;L^2(\Omega)), \ \rtil^t\in L^2(0,T;H^1(\Omega)), \ \rtil^t=0\text{ on }\Gamma_D\}\\
&\subseteq H^1(0,T;L^2(\Omega))+L^2(0,T;H^1(\Omega))
,
\end{aligned}
\end{equation} 
the solution $u$ to the time periodic system
\eqref{JMGT-linearized}, 
\eqref{periodic},
\eqref{bndy}
satisfies the additional regularity
\begin{equation}\label{Ume}
u\in U_{me}:=H^3(0,T;L^2(\Omega))\cap H^2(0,T;H^1(\Omega))\cap H^1(0,T;
H^2_\Delta(\Omega)\cap L^\infty(\Omega))
\end{equation}
\begin{equation}\label{enest_me}
\begin{aligned}
\bar{\mathcal{E}}_{me}(u)\leq C_{me}(
\mathcal{E}_{lo}(u)
&+\taubar^2\|\rtil^\nabla_t\|_{L^2(L^2(\Omega))}^2
+\taubar\|\nabla \rtil^t\|_{L^2(L^2(\Omega))}^2\\
&+\|\rtil^t\|_{L^2(L^2(\Gamma))}^2
+\|\rtil\|_{L^2(L^2(\Omega))}^2
).
\end{aligned}
\end{equation} 
If additionally $\Omega$ is a $C^{1,1/2}$ domain and $\Gamma_D=\emptyset$ then $u\in H^1(0,T;H^{3/2}(\Omega))$.
\item[(iii)] There exist 
$\radius_{lo}>0$, 
$\radius_{me}>0$, 
$\radius_{hi}>0$, $C_{hi}>0$ such that 
for any $\tau\in[0,\taubar]$, and any $b$, $c$, $\alpha$ satisfying \eqref{smallnesscoeffs_lo}, \eqref{smallnesscoeffs_me},
\eqref{smallnesscoeffs_hi}, 
and for any 
\begin{equation}\label{Zhi}
\rtil\in Z_{hi}:=L^2(0,T;H^2_\Delta(\Omega))
\end{equation} 
the solution $u$ to the time periodic system
\eqref{JMGT-linearized}, 
\eqref{periodic},
\eqref{bndy}
satisfies the additional regularity
\begin{equation}\label{Uhi}
u\in U_{hi}:=H^3(0,T;H^1(\Omega))\cap H^2(0,T;H^2_\Delta(\Omega))\cap H^1(0,T;H^3_\Delta(\Omega))
\end{equation} 
and 
\begin{equation}\label{enest_hi}
\bar{\mathcal{E}}_{hi}(u)\leq C_{hi}(
\mathcal{E}_{me}(u)
+\taubar \|\Delta \rtil\|_{L^2(L^2(\Omega))}^2
+\|\Delta \rtil\|_{L^2(H^1(\Omega)^*)}^2
).
\end{equation} 
If additionally $\Omega$ is a $C^{2,1}$ domain and $\partial\Omega=\Gamma_i$ with $\gamma\in W^{1/2,\infty}(\Gamma_i)\cap H^{3/2}(\Gamma_i)$, or $\partial\Omega=\Gamma_N$ or $\partial\Omega=\Gamma_D$, then $u\in H^1(0,T;H^3(\Omega))$.
\end{enumerate}
\end{theorem}

\begin{remark}\label{rem:compare_init}
An inspection of the proof of Theorem~\ref{thm:JMGT-linearized} shows that with initial conditions $u(0)=u_0$, $u_t(0)=u_1$, $\tau u_{tt}(0)=\tau u_2$ in place of periodicity \eqref{periodic}, the energies in \eqref{enest_lo}, \eqref{enest_me}, \eqref{enest_hi} can be replaced by
\begin{equation}\label{energies_initialvalues}
\begin{aligned}
\tilde{\mathcal{E}}_{lo}(u)&=\mathcal{E}_{lo}(u)
+\taubar\tau\|u_{tt}\|_{L^\infty(L^2(\Omega))}^2
+\taubar\|\nabla u_t\|_{L^\infty(L^2(\Omega))}^2
+\|\nabla u\|_{L^\infty(L^2(\Omega))}^2
\\[1ex]
\tilde{\mathcal{E}}_{lo}(u)&=\mathcal{E}_{me}(u)
+\taubar\tau\|\nabla u_{tt}\|_{L^\infty(L^2(\Omega))}^2
+\taubar\|\Delta u_t\|_{L^\infty(L^2(\Omega))}^2
+\|\Delta u\|_{L^\infty(L^2(\Omega))}^2
\\[1ex]
\tilde{\mathcal{E}}_{lo}(u)&=\mathcal{E}_{hi}(u)
+\taubar\tau\|\Delta u_{tt}\|_{L^\infty(L^2(\Omega))}^2
+\taubar\|\nabla\Delta u_t\|_{L^\infty(L^2(\Omega))}^2
+\|\nabla\Delta u\|_{L^\infty(L^2(\Omega))}^2,
\end{aligned}
\end{equation}
while the right hand side in these energy estimates has to be augmented by the respective initial values of these terms, that is, 
by $\taubar\tau\|u_2\|_{L^\infty(L^2(\Omega))}^2
+\taubar\|\nabla u_1\|_{L^\infty(L^2(\Omega))}^2
+\|\nabla u_0\|_{L^\infty(L^2(\Omega))}^2$
by $\taubar\tau\|\nabla u_2\|_{L^\infty(L^2(\Omega))}^2
+\taubar\|\Delta u_1\|_{L^\infty(L^2(\Omega))}^2
+\|\Delta u_0\|_{L^\infty(L^2(\Omega))}^2$
and by $\taubar\tau\|\Delta u_2\|_{L^\infty(L^2(\Omega))}^2$ \\
$+\taubar\|\nabla\Delta u_1\|_{L^\infty(L^2(\Omega))}^2
+\|\nabla\Delta u_0\|_{L^\infty(L^2(\Omega))}^2$
in the low, medium, and high regularity setting, respectively.
This is in line with existing results on the initial value problem for \eqref{JMGT-Westervelt}, \eqref{JMGT-Kuznetsov}, \eqref{JMGT-linearized}, see, e.g., 
\cite{DellOroPata,KLM12_MooreGibson,MarchandMcDevittTriggiani12} for the linear case 
and \cite[Theorem 1.4]{KLP12_JordanMooreGibson}, \cite[Theorems 2.2, 2.3]{JMGT_Neumann}, for the nonlinear cases.
It shows the loss of regularity in the periodic setting and in particular the loss of direct $L^\infty$ in time energy estimates.

Hence an $L^\infty$ bound in time (as needed for controlling stability according to \eqref{stabilitycond}) has to be extracted from $H^1$ in time estimates and Sobolev embeddings.
However, also here we are faced with the fact that arbitrarily large constant offsets are invisible to \eqref{periodic}, so that $\|u\|_{H^1(0,T;X)}\lesssim \|u_t\|_{L^2(0,T;X)}$ fails to hold under periodic boundary conditions, whereas $\|u\|_{H^1(0,T;X)}\lesssim \|u(0)\|_X+\|u_t\|_{L^2(0,T;X)}$ can be well used under  initial conditions. 
\end{remark}

\medskip

\noindent
Secondly, we prove nonlinear well-posedness (see Section~\ref{sec:nonlinearwellposedness}).\\
For the JMGT-Westervelt case \eqref{JMGT-Westervelt} the result can be obtained on a medium regularity level and reads as follows.
\begin{theorem}\label{thm:JMGT-Westervelt} 
Let $T>0$ and $\Omega$ be a $C^{1,1/2}$ domain 
and assume that $\taubar,\,b,\,c$ satisfy \eqref{stabilitycond} with $\alpha=1$, \eqref{ass_coeff} and $\eta\in L^\infty(\Omega)\cap W^{1,6}(\Omega)$.

There exist $\radius_{lo}>0$, $\radius_{JW}>0$, $C_{JW}>0$ such that for any $\tau\in[0,\taubar]$, $b$, $c$, satisfying \eqref{smallnesscoeffs_lo} with $\alpha=1$, and $\|f\|_{Z_{me}}\leq \radius_{JW}$,
the time periodic system
\eqref{JMGT-Westervelt}, 
\eqref{periodic},
\eqref{bndy}
with $\Gamma_D=\emptyset$
has a unique solution  
$u\in U_{me}\cap H^1(0,T;H^{3/2}(\Omega))$
and this solution satisfies the estimate
\begin{equation}\label{enest_W}
\bar{\mathcal{E}}_{me}(u)\leq C_{JW}(
\taubar^2\|f_t\|_{L^2(L^2(\Omega))}^2
+\|f\|_{L^2(L^2(\Omega))}^2
).
\end{equation}
If $\Gamma_a=\emptyset$ then assumption \eqref{smallnesscoeffs_lo} can be skipped and \eqref{enest_W} holds with $\bar{\mathcal{E}}_{me}$ replaced by $\mathcal{E}_{me} $.
\end{theorem}

\begin{remark}\label{rem:regcoeff}
Concerning the required regularity of coefficients we point to the fact that in case $\Gamma_a=\emptyset$, we are able to show well-posedness of \eqref{JMGT-Westervelt} 
\skipKuznetsov{and of \eqref{Kuznetsov}} 
with just $b,\,c\,\in L^\infty(\Omega)$, $\eta\,\in L^\infty(\Omega)\cap W^{1,6}(\Omega)$, thus allowing for jumps in the sound speed and attenualtion coefficient. 
\end{remark}

For the JMGT-Kuznetsov case \eqref{JMGT-Kuznetsov}, including gradient nonlinearity and thus requiring higher regularity, we have
\begin{theorem}\label{thm:JMGT-Kuznetsov}
Let $T>0$ and $\Omega$ be a $C^{2,1}$ domain
and assume that $\taubar,\,b,\,c$ satisfy \eqref{stabilitycond} with $\alpha=1$, \eqref{ass_coeff} and $\eta\in L^\infty(\Omega)\cap W^{1,6}(\Omega)$, $\Delta\eta\in L^2(\Omega)$.

There exist $\radius_{lo}>0$, $\radius_{JK}>0$, $C_{JK}>0$ such that for any $\tau\in[0,\taubar]$, $b$, $c$, satisfying \eqref{smallnesscoeffs_lo}, \eqref{smallnesscoeffs_hi} with $\alpha=1$, and $\|f\|_{Z_{hi}}\leq \radius_{JK}$,
the time periodic system
\eqref{JMGT-Kuznetsov}, 
\eqref{periodic},
\eqref{bndy}
with $\partial\Omega=\Gamma_i$ or $\partial\Omega=\Gamma_N$ or $\partial\Omega=\Gamma_D$
has a unique solution  
$u\in U_{hi}\cap H^1(0,T;H^3(\Omega))$
and this solution satisfies the estimate
\[
\bar{\mathcal{E}}_{hi}(u)\leq C_{JK}(
\taubar^2 \|f_t\|_{L^2(L^2(\Omega))}^2
+\taubar \|\Delta f\|_{L^2(L^2(\Omega))}^2
+\|\Delta f\|_{L^2(H^1(\Omega)^*)}^2
+\|f\|_{L^2(L^2(\Omega))}^2
).
\]
\end{theorem}
Theorems
~\ref{thm:JMGT-Westervelt}, \ref{thm:JMGT-Kuznetsov} comprise the case $\tau=0$ and thus also imply existence and uniqueness of solutions to the time periodic Westervelt and Kuznetsov equations, where for the former we also refer to \cite{periodicWestervelt, periodicWest_2}, while the latter is new.

\skipKuznetsov{
As a by-product, we also find a medium regularity level result on Kuznetsov's equation \eqref{Kuznetsov}, that is, the second order in time equation that results from \eqref{JMGT-Kuznetsov} by formally setting $\tau=0$ (as opposed to the results below that are obtained by taking limits $\tau\searrow0$) in two space dimensions.
\begin{theorem}\label{thm:Kuznetsov}
Let $T>0$ and $\Omega$ be a 
$C^{2,1}$ 
domain 
and assume that $\taubar,\,b,\,c$ satisfy \eqref{stabilitycond} with $\alpha=1$, \eqref{ass_coeff} and $\eta\in L^\infty(\Omega)\cap W^{1,6}(\Omega)$.
\\
There exist $\radius_{lo}>0$, $\radius_{K}>0$, $C_{K}>0$ such that for any $b$, $c$, satisfying \eqref{smallnesscoeffs_lo} with $\alpha=1$, and $\|f\|_{Z_{me}}+***\leq \radius_{K}$,
the time periodic system
\eqref{Kuznetsov}, 
\eqref{periodic},
\eqref{bndy}
with $\partial\Omega=\Gamma_i$ or $\partial\Omega=\Gamma_N$ or $\partial\Omega=\Gamma_D$
has a unique solution  
$u\in U_{me}\cap H^1(0,T;H^{3/2}(\Omega))\cap L^2(0,T;H^3(\Omega))$
and this solution satisfies the estimate
$\mathcal{E}_{me}(u)+\|u\|_{L^2(0,T;H^3(\Omega))}\leq C_{K}(
\taubar^2\|f^\nabla_t\|_{L^2(L^2(\Omega))}^2
+\taubar\|\nabla f^t\|_{L^2(L^2(\Omega))}^2
+\|f^t\|_{L^2(L^2(\Gamma))}^2
+\|f\|_{L^2(L^2(\Omega))}^2
+***
)$.\\
If $\Gamma_a=\emptyset$ then assumption \eqref{smallnesscoeffs_lo} can be skipped.
\end{theorem}
}

\medskip

As a preparation for studying inverse problems in a follow-up paper, we also provide a result on differentiability of the source-to-state map $\mathcal{S}:f\mapsto u$ where $u$ solves \eqref{JMGT-Westervelt} or \eqref{JMGT-Kuznetsov} with \eqref{periodic} and \eqref{bndy}.
\begin{theorem}\label{thm:diff}
Under the conditions of Theorem~\ref{thm:JMGT-Westervelt} with $\Gamma_D=\emptyset$, the mapping $\mathcal{S}:\mathcal{B}^{Z_{me}}_{\radius_{JW}}\to U_{me}\cap H^1(0,T;H^{3/2}(\Omega))$, $f\mapsto u$ satisfying 
\eqref{JMGT-Westervelt},
\eqref{periodic},
\eqref{bndy}
is Fr\'{e}chet differentiable, with derivative $\uld{u}=\mathcal{S}'(f)\uld{f}$ defined by the solution to 
\[
\bigl(\uld{u}+2\eta\, u\, \uld{u}\bigr)_{tt}-c^2\Delta\uld{u}-b\Delta\uld{u}_t+\uld{f}=0,
\]
\eqref{periodic},
\eqref{bndy}.

Under the conditions of Theorem~\ref{thm:JMGT-Kuznetsov} with $\partial\Omega=\Gamma_i$ or $\partial\Omega=\Gamma_N$ or $\partial\Omega=\Gamma_D$, the mapping $\mathcal{S}:\mathcal{B}^{Z_{hi}}_{\radius_{JK}}\to U_{hi}\cap H^1(0,T;H^3(\Omega))$, $f\mapsto u$ satisfying 
\eqref{JMGT-Kuznetsov},
\eqref{periodic},
\eqref{bndy}
is Fr\'{e}chet differentiable, with derivative $\uld{u}=\mathcal{S}'(f)\uld{f}$ defined by the solution to 
\[
\bigl(\uld{u}_t+2\tilde{\eta}\, u_t\, \uld{u}_t+2\nabla u\cdot\nabla\uld{u}\bigr)_t-c^2\Delta\uld{u}-b\Delta\uld{u}_t+\uld{f}=0,
\]
\eqref{periodic},
\eqref{bndy}.
\end{theorem}

\medskip

\noindent
Finally, we consider limits for vanishing relaxation time.
To this end, we consider the spaces induced by the $\tau$ independent parts of the energies
\begin{equation}\label{U0}
\begin{aligned}
&U^0_{lo}:=H^2(0,T;L^2(\Omega))\cap H^1(0,T;H^1(\Omega))\\
&U^0_{me}:=H^2(0,T;H^1(\Omega))\cap H^1(0,T;H^2_\Delta(\Omega))\\
&U^0_{hi}:=H^2(0,T;H^2_\Delta(\Omega))\cap H^1(0,T;H^3_\Delta(\Omega)).
\end{aligned}
\end{equation}

We obtain weak convergence results for the linear MGT, the JMGT-Westervelt, and the JMGT-Kuznetsov case.

\begin{theorem}\label{thm:taulimit} 
In the linear setting, under the assumptions of Theorem~\ref{thm:JMGT-linearized} (i), (ii), or (iii), 
the family $(u^\tau)_{\tau\in(0,\taubar]}$ of solutions to 
\eqref{JMGT-linearized},
\eqref{periodic},
\eqref{bndy}
converges weakly in $U^0_{lo}$, $U^0_{me}$, or $U^0_{hi}$, respectively, to the solution $\bar{u}$ of 
\eqref{JMGT-linearized} with $\tau=0$,
\eqref{periodic},
\eqref{bndy}.

In the nonlinear setting, under the assumptions of Theorem~\ref{thm:JMGT-Westervelt} / \ref{thm:JMGT-Kuznetsov}
the family $(u^\tau)_{\tau\in(0,\taubar]}$ of solutions to 
$\begin{cases}\text{\eqref{JMGT-Westervelt}} \\ \text{\eqref{JMGT-Kuznetsov}}\end{cases}$ 
with \eqref{periodic}, \eqref{bndy}
converges weakly in $\begin{cases}U^0_{me} \\ U^0_{hi}\end{cases}$ to the solution $\bar{u}$ of 
$\begin{cases}\text{\eqref{Westervelt}} \\ \text{\eqref{Kuznetsov}}\end{cases}$ 
with \eqref{periodic}, \eqref{bndy}.
\end{theorem}

Given the estimates in Theorems~\ref{thm:JMGT-linearized}, \ref{thm:JMGT-Westervelt}, \ref{thm:JMGT-Kuznetsov} and considering the $\tau$ independent parts of the energies, the proof of Theorem~\ref{thm:taulimit} follows analogosly to the one of \cite[Theorem 7.1]{JMGT}; see also \cite{BongartiCharoenphonLasiecka20,JMGT_Neumann}.
Uniqueness of a solution to the limiting equation, as required for the subsequence-subsequence argument in this proof, follows from 
the case $\tau=0$ in Theorems~\ref{thm:JMGT-linearized}, \ref{thm:JMGT-Westervelt}, \ref{thm:JMGT-Kuznetsov}.

\skipKuznetsov{
The difference between the results on Kuznetsov's equation contained in Theorems~\ref{thm:Kuznetsov} and \ref{thm:taulimit} lies in their level of regularity, which is ``medium'' in the former and ``high'' in the latter case.
}

\section{Linear well-posedness; proof of Theorem~\ref{thm:JMGT-linearized}}
\label{sec:linearwellposedness}
To study well-posedness of the linearized equation \eqref{JMGT-linearized}, for conciseness of notation we write it as
\begin{equation}\label{JMGT_linearized_rtil}
(\tau \partial_t^3 + \alpha\partial_t^2) u+ (b\partial_t+c^2\text{id})(-\Delta) u + \tilde{r}=0
\text{ on }\Omega\times(0,T)
\end{equation}

Its well-posendess analysis with energy estimates is the fundamental building block for proving the self-mapping and contractivity estimates in a fixed point proof, but also allows for a proof of Fr\'{e}chet differentiability of the source-to-state map $\mathcal{S}$ in Section~\ref{sec:nonlinearwellposedness}.

\subsubsection*{Variational formulation}
We will carry out our analysis on three levels of regularity with respect to the solution (cf. \eqref{Ulo}, \eqref{Ume}, \eqref{Uhi}) and test spaces
\begin{eqnarray}
&&V_{lo}=\{v\in L^2(0,T;H^1(\Omega))\, : \, 
v=0 \mbox{ on }\Gamma_D, \
v(T)=v(0), \, v_t(T)=v_t(0)\}
\label{Vlo}\\ 
&&V_{me}=\{v\in L^2(0,T;L^2(\Omega))\, : \, -\Delta v\in L^2(0,T;L^2(\Omega))\,:\,
\nonumber \\
&&\qquad\qquad\qquad\qquad\partial_\nu v+(\beta \partial_t+\gamma\text{id}) v =0 \mbox{ on }\Gamma_a\cup\Gamma_i\cup\Gamma_N,\
v=0 \mbox{ on }\Gamma_D, 
\nonumber \\ 
&&\qquad\qquad\qquad\qquad v(T)=v(0), \, v_t(T)=v_t(0)\}
\label{Vme}\\ 
&&V_{hi}=\{v\in V_{me}\, : \, \Delta v\in V_{me}\}
\label{Vhi}
\end{eqnarray}

Testing \eqref{JMGT_linearized_rtil} with $v\in V_{lo}$ we obtain a formulation of \eqref{JMGT_linearized_rtil} as a \underline{v}ariational \underline{e}quation 
\begin{equation}\label{ve_lo}
\begin{aligned}
&u\in U_{lo}, \tau u_{tt}(T)=\tau u_{tt}(0)\text{ and for all }v\in V_{lo}\\
&(ve)\quad
\left\{\begin{array}{l}
{\displaystyle
\int_0^T\Bigl\{\int_\Omega \Bigl( \bigl((\tau \partial_t^3 + \alpha\partial_t^2) u+ \rtil\bigr)\, v
+\nabla(b\partial_t+c^2\text{id}) u\cdot\nabla v\Bigr)\, dx 
}\\
{\displaystyle \qquad
+\int_{\Gamma_a\cup\Gamma_i} (\beta \partial_t+\gamma\text{id})(b\partial_t+c^2\text{id})u\, v\, dS
\Bigr\}\, dt=0
}\end{array}\right.
\end{aligned}
\end{equation}

\subsubsection*{Galerkin approximation} 
To account for periodicity as well as the absorbing boundary condition (which combines zero and first order time derivative values) we use periodic space-time Galerkin ansatz functions 
\[
U_{MJ}=\{(t,x)\mapsto\Re\Bigl(\sum_{m=1}^M\exp(\imath m\omega t)\hat{v}_m(x)\Bigr)\, : \, \hat{v}_m\in \text{span}(\phi_{m\,1},\ldots,\phi_{m\,J})\}
\]
with $\omega=\frac{2\pi}{T}$ and $\phi_{m\,j}$ eigenfunctions of the Laplacian with frequency dependent impedance boundary conditions
\[
\begin{aligned}
&-\Delta \phi_{m\,j}=\lambda_{m\,j} \phi_{m\,j}, \text{ in }\Omega\\
&\partial_\nu \phi_{m\,j}+\imath m\omega\beta\phi_{m\,j}+\gamma \phi_{m\,j} =0 \mbox{ on }\Gamma_a\cup\Gamma_i\cup\Gamma_N,\
\phi_{m\,j}=0 \mbox{ on }\Gamma_D
\end{aligned}
\] 
By this construction we have 
\[
\begin{aligned}
&U_{MJ}\subseteq V_{hi}\subseteq V_{me} \text{ for all }N,\, J\, \in\mathbb{N} \text{ and }\\ 
&\overline{\bigcup_{M\in\mathbb{N}}\bigcup_{J\in\mathbb{N}}U_{MJ}}^{V_{hi}}=V_{hi}, \quad 
\overline{\bigcup_{M\in\mathbb{N}}\bigcup_{J\in\mathbb{N}}U_{MJ}}^{V_{me}}=V_{me}
\end{aligned}
\]


We will use three equivalent Galerkin formulations for the three stages of regularity
\begin{itemize}
\item[(i)] low; 
\commBK{(to be used for linear (J)MGT and as a foundation for all higher order estimates)\\}
by testing \eqref{ve_lo} with $v$,\commBK{\footnote{possibly also consider testing with $\tfrac{1}{c^2}v$ in order to allow for $c\in L^\infty(\Omega)$ (along with the assumption $\tfrac{b}{c^2}$ smooth enough); however, then consistency of the three Galerkin discretizations according to \eqref{equivalence} gets lost, so this testing strategy only makes sense for obtaining a low level energy estimate for the linear case.}} 
\begin{equation}\label{var_lo}
\begin{aligned}
&u\in U_{MJ}\subseteq U_{lo}, \tau u_{tt}(T)=\tau u_{tt}(0)\text{ and for all }v\in U_{MJ}\subseteq V_{lo}\\
&\int_0^T\Bigl\{\int_\Omega \Bigl( \bigl((\tau \partial_t^3 + \alpha\partial_t^2) u+ \rtil\bigr)\, v
+\nabla(b\partial_t+c^2\text{id}) u\cdot\nabla v\Bigr)\, dx 
\\ &\qquad
+\int_{\Gamma_a\cup\Gamma_i} (\beta \partial_t+\gamma\text{id})(b\partial_t+c^2\text{id})u\, v\, dS
\Bigr\}\, dt=0
\end{aligned}
\end{equation}
\item[(ii)] medium;
\commBK{ (for JMGT-Westervelt in pressure or potential formulation; for $\tau=0$, i.e.,  Kuznetsov); allows for $b,c\in L^\infty$\\}
by testing \eqref{ve_lo} with $-\Delta v$, 
\begin{equation}\label{var_me}
\begin{aligned}
&u\in U_{MJ}\subseteq U_{me}, \tau u_{tt}(T)=\tau u_{tt}(0)\text{ and for all }v\in U_{MJ}\subseteq V_{me}\\
&\int_0^T\Bigl\{\int_\Omega \Bigl( \nabla\bigl(\tau \partial_t^3 + \alpha\partial_t^2) u\bigr)\cdot\nabla v
+\bigl((b\partial_t+c^2\text{id})\Delta u- \rtil \bigr)\,\Delta v\Bigr)\, dx \\
&\qquad+\int_{\Gamma_a\cup\Gamma_i} (\tau \partial_t^3 + \alpha\partial_t^2) u\, (\beta \partial_t+\gamma\text{id})v\, dS
\Bigr\}\, dt=0
\end{aligned}
\end{equation}
\item[(iii)] high; 
\commBK{(for JMGT-Kuznetsov, that is, with gradient term and $\tau>0$);\\} 
by testing \eqref{ve_lo} with $(-\Delta)^2 v$, 
\begin{equation}\label{var_hi}
\begin{aligned}
&u\in U_{MJ}\subseteq U_{hi}, \tau u_{tt}(T)=\tau u_{tt}(0)\text{ and for all }v\in U_{MJ}\subseteq V_{hi}\\
&\int_0^T\Bigl\{
\int_\Omega \Bigl(\Delta\bigl((\tau \partial_t^3 + \alpha\partial_t^2) u\bigr)\, \Delta v+\nabla\bigl((b\partial_t+c^2\text{id})\Delta u\bigr)\,\nabla \Delta v + \Delta \rtil\, \Delta v\Bigr)\, dx\\
&\qquad+\int_{\Gamma_a\cup\Gamma_i} 
\bigl((\beta \partial_t+\gamma\text{id})(b\partial_t+c^2\text{id})\Delta u \bigr)\, \Delta v\, dS
\Bigr\}\, dt=0
\end{aligned}
\end{equation}
\commBK{
where the $\rtil$ term can alternatively be represented using integration by parts
\[
\int_\Omega \Delta \rtil\, \Delta v\, dx =
-\int_\Omega \nabla \rtil\, \nabla \Delta v\, dx +\int_{\partial\Omega} \partial_\nu \rtil\, \Delta v\, dS
\]
}

\commBK{
Alternative derivation of \eqref{var_hi}: Apply $\Delta $ to PDE and test with $\Delta v$:
\[
\int_0^T
\int_\Omega \Bigl(\Delta\bigl((\tau \partial_t^3 + \alpha\partial_t^2) u\bigr)\,\Delta v-\Delta\bigl((b\partial_t+c^2\text{id})\Delta u- \rtil \bigr)\, \Delta v\Bigr)\, dx\, dt =0
\]
}
\end{itemize}
Equivalence of the three Galerkin formulations \eqref{var_lo}, \eqref{var_me}, \eqref{var_hi} follows from  
\begin{equation}\label{equivalence}
v\in U_{MJ} \ \Leftrightarrow \ -\Delta v \in U_{MJ} \ \Leftrightarrow \ (-\Delta)^2 v \in U_{MJ}.
\end{equation}
As a consequence we can use them to derive low, medium and high order energy identities as well as subsequently estimates on the Galerkin discretized level.

Existence of a solution $u_{MJ}\in U_{MJ}$ to \eqref{var_lo}, (and thus of \eqref{var_me}, \eqref{var_hi}) follows from uniqueness (which is a consequence of the upcoming energy estimates) since  $U_{MJ}$ is finite dimensional.
\footnote{Thus, unlike \cite{periodicWestervelt,periodicWest_2}, we do not invoke periodic ODE theory such as the Floquet-Lyapunov Theorem, here.}

\subsubsection*{Energy estimates} 
In each of the three cases (i), (ii), (iii), we use $v=\taubar u_{MJ\,tt}+\sigma u_{MJ\,t}+\rho u_{MJ}$ with a proper choice of $\sigma$ and $\rho$.
\footnote{Note that testing with $\rho u$ (while being used to derive equipartition of energy in the initial value problem setting, cf. e.g., \cite{KLM12_MooreGibson,KLP12_JordanMooreGibson}) is here required for obtaining the respective full norm wrt time in the periodic setting. Indeed, combining e.g., an estimate of $u_t$ with time periodic boundary conditions on $u$ does not capture the full $H^1$ in time norm of $u$, since periodicity admits arbitrarily large time constant offsets.}
For easier readability we will skip the indices $M$ and $J$ in the energy estimates.

Using the fact that terms of the form $\int_0^T\tfrac{d}{dt}(\ldots)\, dt$ vanish due to periodicity and integrating by parts with respect to time (where again, due to periodicity we can skip initial-end terms and where we use the identity $\alpha u_t u_{tt}=\alpha\frac12\frac{d}{dt}(u_t^2)=\frac12(\frac{d}{dt}(\alpha u_t^2)-\alpha_t u_t^2)$) yields
\begin{itemize}
\item the low order energy identity
\commBK{
\[
\hspace*{-3cm}\begin{aligned}
0=&\int_0^T\Bigl\{\int_\Omega \Bigl( \bigl((\tau \partial_t^3 + \alpha\partial_t^2) u+ \rtil\bigr)\, (\taubar u_{tt}+\sigma u_t+\rho u)
+\nabla\bigl(b\partial_t+c^2\text{id}) u\bigr)\cdot\nabla (\taubar u_{tt}+\sigma u_t+\rho u)\Bigr)\, dx \\
&\qquad+\int_{\Gamma_a\cup\Gamma_i} (\beta \partial_t+\gamma)(b\partial_t+c^2\text{id})u\, (\taubar u_{tt}+\sigma u_t+\rho u)\, dS
\Bigr\}\, dt\\
\end{aligned}
\]
}
\begin{equation}\label{enid_lo}
\begin{aligned}
0=&\int_0^T\Bigl\{\int_\Omega \Bigl( (\taubar\alpha-\tau\sigma) u_{tt}^2 -(\rho\alpha+\tfrac{\sigma}{2}\alpha_t)u_t^2-\rho\alpha_t u_t u  + \rtil\, (\taubar u_{tt}+\sigma u_t+\rho u)\\
&\qquad
-\nabla\cdot\bigl(\nabla b \,u_t+\nabla c^2 \, u\bigr)\,(\taubar u_{tt}+\sigma u_t+\rho u)
+(\sigma b-\taubar c^2)|\nabla u_t|^2+\rho c^2 |\nabla u|^2\Bigr)\, dx \\
&\qquad+\int_{\Gamma_a\cup\Gamma_i} \bigl(\taubar\beta b u_{tt}^2 
+(\beta(\sigma c^2-\rho b)+\gamma(\sigma b-\taubar c^2))
u_t^2+\rho\gamma c^2 u^2\bigr)\, dS\\
&\qquad+\int_{\partial\Omega} \bigl(\partial_\nu b \,u_t+\partial_\nu c^2 \, u\bigr)\,(\taubar u_{tt}+\sigma u_t+\rho u)
\Bigr\}\, dt
\end{aligned}
\end{equation}
\item the medium order energy identity
\commBK{
\[
\hspace*{-3cm}\begin{aligned}
0=&\int_0^T\Bigl\{\int_\Omega \Bigl( \nabla\bigl((\tau \partial_t^3 + \alpha\partial_t^2) u\bigr)\cdot\nabla (\taubar u_{tt}+\sigma u_t+\rho u)
+\bigl((b\partial_t+c^2\text{id})\Delta u- \rtil\bigr)\,\Delta (\taubar u_{tt}+\sigma u_t+\rho u)\Bigr)\, dx \\
&\qquad+\int_{\Gamma_a\cup\Gamma_i} (\tau \partial_t^3 + \alpha\partial_t^2) u\, (\beta \partial_t+\gamma)(\taubar u_{tt}+\sigma u_t+\rho u)\, dS
\Bigr\}\, dt
\end{aligned}
\]
}
\begin{equation}\label{enid_me}
\begin{aligned}
0=&\int_0^T\Bigl\{\int_\Omega \Bigl( \nabla\alpha u_{tt}\cdot\nabla (\taubar u_{tt}+\sigma u_t+\rho u)
+(\taubar\alpha-\tau\sigma)|\nabla u_{tt}|^2\\
&\qquad\qquad -(\rho\alpha+\tfrac{\sigma}{2}\alpha_t) |\nabla u_t|^2-\rho\alpha_t\nabla u_t\cdot\nabla u\\
&\qquad\qquad+(\sigma b-\taubar c^2)(\Delta u_t)^2 + \rho c^2 (\Delta u)^2
- \rtil \,\Delta (\taubar u_{tt}+\sigma u_t+\rho u)\Bigr)\, dx \\
&\qquad+\int_{\Gamma_a\cup\Gamma_i} \Bigl(
\tau\taubar\beta u_{ttt}^2+
\bigl(\beta(\sigma\alpha-\tau\rho-\tfrac{\taubar}{2}\alpha_t)+\gamma(\taubar\alpha-\sigma\tau)\bigr)
u_{tt}^2\\
&\qquad\qquad\qquad-\bigl(\rho\gamma\alpha+\tfrac12(\rho\beta+\sigma\gamma)\alpha_t\bigr) u_t^2
-\rho\gamma\alpha_t u_t u\Bigr)\, dS
\Bigr\}\, dt
\end{aligned}
\end{equation}
where we split $\rtil=\rtil^\nabla+\rtil^t$ and assume that $\rtil^t$ satisfies homogeneous Dirichlet boundary conditions;\footnote{Later on we will set $\rtil^t:=r^t[u]$ which inherits homogeneous Dirichlet boundary conditions from $u$}
using integration by parts, as well as the absorbing, impedance and Neumann boundary conditions that $u_{tt}$ satisfies, we have
\begin{equation}\label{est_ru_me}
\begin{aligned}
&\int_0^T\int_\Omega \rtil\,\taubar\Delta u_{tt}
= -\taubar\Bigl(\int_0^T\int_\Omega (\rtil^\nabla_t\Delta u_t+\nabla \rtil^t\cdot\nabla u_{tt})
+\int_0^T\int_\Gamma \rtil^t(\beta u_{ttt}+\gamma u_{tt})\, dS\, dt\Bigr)\\
&\leq \taubar \|\rtil^\nabla_t\|_{L^2(L^2(\Omega))} \|\Delta u_t\|_{L^2(L^2(\Omega))}
+\sqrt{\taubar} \|\nabla \rtil^t\|_{L^2(L^2(\Omega))} \sqrt{\taubar}\|\nabla u_{tt}\|_{L^2(L^2(\Omega))}\\
&\qquad+\|\rtil^t\|_{L^2(L^2(\Gamma))}(\taubar\|\beta u_{ttt}\|_{L^2(L^2(\Gamma_a))}+\taubar\|\gamma u_{tt}\|_{L^2(L^2(\Gamma))})\\
&\leq \frac{\upsilon}{2}\Bigl(\|\Delta u_t\|_{L^2(L^2(\Omega))}^2+\taubar\|\nabla u_{tt}\|_{L^2(L^2(\Omega))}^2
+\taubar^2\|\beta u_{ttt}\|_{L^2(L^2(\Gamma_a))}^2+\taubar^2\|\gamma u_{tt}\|_{L^2(L^2(\Gamma))}^2\Bigr)\\
&\qquad+\frac{1}{2\upsilon}\Bigl(\taubar^2 \|\rtil^\nabla_t\|_{L^2(L^2(\Omega))}^2+\taubar\|\nabla \rtil^t\|_{L^2(L^2(\Omega))}^2 +\|\rtil^t\|_{L^2(L^2(\Gamma))}^2\Bigr).
\end{aligned}
\end{equation}
\commBK{
Alternatively, 
\[
\begin{aligned}
&\leq \frac{\upsilon}{2}\Bigl(\|\Delta u_t\|_{L^2(L^2(\Omega))}^2+\taubar\|\nabla u_{tt}\|_{L^2(L^2(\Omega))}^2
+\taubar\|\beta u_{ttt}\|_{L^2(L^2(\Gamma_a))}^2+\taubar\|\gamma u_{tt}\|_{L^2(L^2(\Gamma))}^2\Bigr)\\
&\qquad+\frac{1}{2\upsilon}\Bigl(\taubar^2 \|r^\nabla[u]\|_{L^2(L^2(\Omega))}^2+\taubar\|\nabla r^t[u]\|_{L^2(L^2(\Omega))}^2 +\taubar\|r^t[u]\|_{L^2(L^2(\Gamma))}^2\Bigr).
\end{aligned}
\]
where we could estimate the boundary term on $\Gamma$ (actually even $\taubar\|\beta u_{ttt}\|_{L^2(L^2(\Gamma_a))}^2+\|\gamma u_{tt}\|_{L^2(L^2(\Gamma))}^2$) by $\mathcal{E}_{lo}[u_t]$.
}
\item the high order energy identity
\commBK{
\[
\hspace*{-3cm}\begin{aligned}
0=&\int_0^T\Bigl\{
\int_\Omega \Bigl(\Delta\bigl((\tau \partial_t^3 + \alpha\partial_t^2) u\bigr)\, \Delta (\taubar u_{tt}+\sigma u_t+\rho u)+\nabla\bigl((b\partial_t+c^2\text{id})\Delta u- \rtil\bigr)\,\nabla \Delta (\taubar u_{tt}+\sigma u_t+\rho u)\Bigr)\, dx\\
&\qquad+\int_{\Gamma_a\cup\Gamma_i} 
\bigl((\beta \partial_t+\gamma)(b\partial_t+c^2\text{id})\Delta u +\partial_\nu \rtil\bigr)\, \Delta (\taubar u_{tt}+\sigma u_t+\rho u)\, dS
\Bigr\}\, dt
\end{aligned}
\]
}
\begin{equation}\label{enid_hi}
\begin{aligned}
0=&\int_0^T\Bigl\{\int_\Omega \Bigl( (\Delta\alpha u_{tt}+2\nabla\alpha\cdot \nabla u_{tt})\cdot\Delta (\taubar u_{tt}+\sigma u_t+\rho u)
+(\taubar\alpha-\tau\sigma)|\Delta u_{tt}|^2\\
&\qquad\qquad 
-(\rho\alpha+\tfrac{\sigma}{2}\alpha_t) |\Delta u_t|^2
-\rho\alpha_t \Delta u_t\, \Delta u\\
&\qquad\qquad+(\sigma b-\taubar c^2)(\nabla\Delta u_t)^2 + \rho c^2(\nabla\Delta u)^2\\
&\qquad\qquad
-\bigl(\nabla b\, \Delta u_{tt}+\nabla c^2\, \Delta u_t\bigr)\cdot\nabla\Delta (\taubar u_t+\sigma u)
+\rho\bigl(\nabla b\, \Delta u_t+\nabla c^2\, \Delta u\bigr)\cdot\nabla\Delta u
\\&\qquad\qquad 
+\Delta \rtil \,\Delta (\taubar u_{tt}+\sigma u_t+\rho u)
\Bigr)\, dx \\
&\qquad+\int_{\Gamma_a\cup\Gamma_i} \bigl(
\taubar\beta b (\Delta u_{tt})^2
+
(\beta(\sigma c^2-\rho b)+\gamma(\sigma b-\taubar c^2))
(\Delta u_t)^2  
+\rho\gamma c^2  (\Delta u)^2
\bigr)\, dS
\Bigr\}\, dt
\end{aligned}
\end{equation}
\end{itemize}

To derive energy estimates from \eqref{enid_lo}, \eqref{enid_me}, \eqref{enid_hi}, we enforce positivity of the multipliers of the seminorms appearing here by imposing the conditions
\begin{equation}\label{condsigmarho}
\overline{\tau}\tfrac{c^2}{b}<\sigma<\alpha, \quad
\rho\tfrac{b}{c^2}<\sigma, \quad
\rho\leq\tfrac{\sigma}{\tau}\alpha 
\end{equation}
on $\sigma>0$ and $\rho>0$, which can be achieved due to \eqref{stabilitycond} with the choices
\[
\sigma:=\inf_{(0,T)\times\Omega}\tfrac12(\taubar\tfrac{c^2}{b}+\alpha)>0 \text{ and $\rho>0$ small enough (but independent of $\tau$).}
\]
\commBK{Note that we don't want to have $\sigma,\rho\to0$ as $\tau\to0$ in view of our plan to take the limit $\tau\to0$ in Theorem~\ref{thm:taulimit}.}

Moreover, we make use of \eqref{PF}, that, e.g., in the low order energy identity allows us to estimate
\[
\begin{aligned}
&\int_0^T\int_\Omega (\rho\alpha+\tfrac{\sigma}{2}\alpha_t)u_t^2\, dx\, dt\\
&\leq  \|\rho\alpha+\tfrac{\sigma}{2}\alpha_t\|_{L^\infty(L^{q^*}(\Omega)}
(C_{H^1\to L^{2q}}^\Omega)^2 
C_{1\,PF}(\|\nabla u_t\|_{L^2(L^2(\Omega))}^2+\|\nabla u_t\|_{L^2(L^2(\Gamma_a\cup\Gamma_i))}^2)\\
&\leq 
\int_0^T\Bigl\{\int_\Omega \tfrac12(\sigma b-\taubar c^2)|\nabla u_t|^2 \, dx + \tfrac12\int_{\Gamma_a\cup\Gamma_i} (\beta(\sigma c^2-\rho b)+\gamma(\sigma b-\taubar c^2)) u_t^2\, dS\Bigr\}\, dt,
\end{aligned}
\]
provided 
\begin{equation}\label{smallalphat}
\|\rho\alpha+\tfrac{\sigma}{2}\alpha_t\|_{L^\infty(L^{q^*}(\Omega)} 
\leq \frac{\min\{\inf_\Omega \sigma b-\taubar c^2,\,
\inf_{\Gamma_a\cup\Gamma_i}\beta(\sigma c^2-\rho b)+\gamma(\sigma b-\taubar c^2)\}}{2(C_{H^1\to L^{2q}}^\Omega)^2 C_{1\,PF}}
\end{equation}

The further terms that arise due to varying coefficients in \eqref{enid_lo}, \eqref{enid_me}, \eqref{enid_hi} can be estimated as follows
\begin{itemize}
\item in the low order energy identity:
\[\begin{aligned}
&\int_0^T\int_\Omega -\nabla\cdot\bigl(\nabla b \,u_t+\nabla c^2 \, u\bigr)\,(\taubar u_{tt}+\sigma u_t+\rho u)\, dx\, dt\\
&\leq \Bigl(
\max\{\|\nabla b\|_{L^\infty(L^2(\Omega))},\,\sqrt{T}\|\nabla c^2\|_{L^2(L^2(\Omega))}\} \\
&\qquad+C_{H^1\to L^p}\max\{\|\Delta b\|_{L^\infty(L^{\hat{p}}(\Omega))},\,\sqrt{T}\|\Delta c^2\|_{L^2(L^{\hat{p}}(\Omega))}\Bigr) \|u\|_{H^1(H^1(\Omega))}\\
&\qquad\qquad\Bigl(\taubar \|u_{tt}\|_{L^2(L^2(\Omega))}+\sigma\|u_t\|_{L^2(L^2(\Omega))}+\rho\|u\|_{L^2(L^2(\Omega))} \Bigr)
\end{aligned}\]
for $1-\frac{d}{2}\geq-\frac{d}{p}$, $\hat{p}=\frac{2p}{p-2}$
and we make the (realistic) assumption $\partial_\nu b=0$, $\partial_\nu c^2=0$ on $\partial\Omega$ to remove the boundary term containing these normal derivatives.
\item in the medium order energy identity:
\[\begin{aligned}
&\int_0^T\int_\Omega \nabla\alpha u_{tt}\cdot\nabla (\taubar u_{tt}+\sigma u_t+\rho u)\, dx\, dt\\
&\leq \|\nabla\alpha\|_{L^\infty(L^{\hat{p}}(\Omega))} \|u_{tt}\|_{L^2(L^p(\Omega))}
\Bigl(\taubar \|\nabla u_{tt}\|_{L^2(L^2(\Omega))}+\sigma\|\nabla u_t\|_{L^2(L^2(\Omega))}+\rho\|\nabla u\|_{L^2(L^2(\Omega))} \Bigr)
\end{aligned}\]
Note that no derivatives of $b$, $c^2$ are needed here, which will allow us to work with just $L^\infty(\Omega)$ attenuation and sound speed in certain cases. 
\item in the high order energy identity:
\[\begin{aligned}
&\int_0^T\int_\Omega (\Delta\alpha u_{tt}+2\nabla\alpha\cdot \nabla u_{tt})\cdot\Delta (\taubar u_{tt}+\sigma u_t+\rho u)\, dx\, dt\\
&\leq \Bigl(\|\Delta\alpha\|_{L^\infty(L^2(\Omega))} \|u_{tt}\|_{L^2(L^\infty(\Omega))}
+ 2\|\nabla\alpha\|_{L^\infty(L^{\hat{p}}(\Omega))} \|\nabla u_{tt}\|_{L^2(L^p(\Omega))}\Bigr)\\
&\qquad\qquad\Bigl(\taubar \|\Delta u_{tt}\|_{L^2(L^2(\Omega))}+\sigma\|\Delta u_t\|_{L^2(L^2(\Omega))}+\rho\|\Delta u\|_{L^2(L^2(\Omega))} \Bigr)
\end{aligned}\]
\[\begin{aligned}
&\int_0^T\int_\Omega 
\bigl(\nabla b\, \Delta u_{tt}+\nabla c^2\, \Delta u_t\bigr)\cdot\nabla\Delta (\taubar u_t+\sigma u)
\, dx\, dt\\
&\leq \Bigl(\|\nabla b\|_{L^\infty(\Omega)} \|\Delta u_{tt}\|_{L^2(L^2(\Omega))}
+ \|\nabla c^2\|_{L^\infty(\Omega)} \|\Delta u_t\|_{L^2(L^2(\Omega))} \Bigr)
\\
&\qquad\qquad\Bigl(\taubar \|\nabla\Delta u_t\|_{L^2(L^2(\Omega))}+\sigma\|\nabla\Delta u\|_{L^2(L^2(\Omega))} \Bigr)
\end{aligned}\]
\end{itemize}
By imposing smallness of the derivatives of $\alpha$, $b$, $c^2$ in the norms appearing here, we can absorb them into the positive terms on the left hand side of \eqref{enid_lo}, \eqref{enid_me}, \eqref{enid_hi}.

\medskip

To obtain an estimate on the third time derivative of $u$, we use the PDE (more precisely, we test the Galerkin discretizations with $u_{MJ\,ttt}$ and apply the Cauchy-Schwarz inequality) to obtain estimates of the type 
\[
\tau\|u_{ttt}\|_{L^2(X)}\leq \|\alpha u_{tt}+(b\partial_t+c^2\text{id})(-\Delta) u + \tilde{r}\|_{L^2(X)}
\]
with different spaces $X$ depending on the already obtained regularity of $\alpha u_{tt}+(b\partial_t+c^2\text{id})(-\Delta) u + \tilde{r}$

Thus, the following energy estimates result:
\begin{itemize}
\item from \eqref{enid_lo}:
\begin{equation}\label{energies_lo}
\begin{aligned}
\mathcal{E}_{lo}(u)&=
\taubar\tau^2\|u_{ttt}\|_{L^2(H^1(\Omega)^*)}^2+
\taubar\|u_{tt}\|_{L^2(L^2(\Omega))}^2 + \|u\|_{H^1(H^1(\Omega))}^2\\
&\qquad
+ \taubar \|u_{tt}\|_{L^2(L^2(\Gamma_a))}^2+\|\sqrt{\gamma} u\|_{H^1(L^2(\Gamma))}^2\\
&\lesssim \taubar\|\rtil\|_{L^2(L^2(\Omega))}^2+\|\rtil\|_{L^2(H^1(\Omega)^*)}^2
\end{aligned}
\end{equation}
\commBK{
\[
\mathcal{E}_{lo}(u_t)
\lesssim \taubar\|\alpha_t u_{tt}+(\rtil)_t\|_{L^2(L^2(\Omega))}^2+\|\alpha_t u_{tt}+(\rtil)_t\|_{L^2(H^1(\Omega)^*)}^2
\quad \text{(by differentiating the PDE wrt time)}
\]
}
\item from \eqref{enid_lo}+$\chi$\eqref{enid_me} with 
\[\frac{1}{\chi}=2\max\left\{
\left\|\frac{\rho\alpha+\tfrac{\sigma}{2}\alpha_t}{\sigma b-\taubar c^2}\right\|_{L^\infty(L^\infty(\Omega))},\,
\left\|\frac{\rho\gamma\alpha+\tfrac12(\rho\beta+\sigma\gamma)\alpha_t}{\beta(\sigma c^2-\rho b)+\gamma(\sigma b-\taubar c^2)}\right\|_{L^\infty(L^\infty(\Omega))}\right\}
\]
\begin{equation}\label{energies_me}
\begin{aligned}
\bar{\mathcal{E}}_{me}(u)&=
\mathcal{E}_{lo}(u)+
\taubar\tau^2\|u_{ttt}\|_{L^2(L^2(\Omega))}^2+
\taubar\|u_{tt}\|_{L^2(H^1(\Omega))}^2 + \|\Delta u\|_{H^1(L^2(\Omega))}^2\\
&\qquad+ \taubar\tau \|u_{ttt}\|_{L^2(L^2(\Gamma_a))}^2+ \|\sqrt{\gamma} u\|_{H^2(L^2(\Gamma))}^2\\
&\lesssim 
\mathcal{E}_{lo}(u)
+\taubar^2\|\rtil^\nabla_t\|_{L^2(L^2(\Omega))}^2
+\taubar\|\nabla \rtil^t\|_{L^2(L^2(\Omega))}^2
+\|\rtil^t\|_{L^2(L^2(\Gamma))}^2
+\|\rtil\|_{L^2(L^2(\Omega))}^2
\end{aligned}
\end{equation}
\commBK{ or 
\[
\lesssim 
\mathcal{E}_{lo}(u)+\taubar \mathcal{E}_{lo}(u_t)
+\taubar^2\|\rtil^\nabla_t\|_{L^2(L^2(\Omega))}^2
+\taubar\|\nabla \rtil^t\|_{L^2(L^2(\Omega))}^2
+\taubar\|\rtil^t\|_{L^2(L^2(\Gamma))}^2
+\|\rtil\|_{L^2(L^2(\Omega))}^2
\]
}
\item from \eqref{enid_lo}+$\chi$(\eqref{enid_me}+$\lambda$\eqref{enid_hi}) with 
\[
\frac{1}{\lambda}=2\left\|\frac{\rho\alpha+\tfrac{\sigma}{2}\alpha_t}{\sigma b-\taubar c^2}\right\|_{L^\infty(L^\infty(\Omega))}
\]
\begin{equation}\label{energies_hi}
\begin{aligned}
\bar{\mathcal{E}}_{hi}(u)&=
\mathcal{E}_{me}(u)+
\taubar\tau^2\|u_{ttt}\|_{L^2(H^1(\Omega)^*)}^2+
\taubar\|\Delta u_{tt}\|_{L^2(L^2(\Omega))}^2 + \|\nabla\Delta u\|_{H^1(L^2(\Omega))}^2\\
&\qquad
+ \taubar \|\Delta u_{tt}\|_{L^2(L^2(\Gamma_a))}^2 + \|\gamma \Delta u\|_{H^1(L^2(\Gamma))}^2\\
&\lesssim
\mathcal{E}_{me}(u)
+\taubar \|\Delta \rtil\|_{L^2(L^2(\Omega))}^2
+\|\Delta \rtil\|_{L^2(H^1(\Omega)^*)}^2
\end{aligned}
\end{equation}
\commBK{
or 
\[\lesssim 
\mathcal{E}_{me}(u)
+\taubar \|\Delta \rtil\|_{L^2(L^2(\Omega))}^2
+\|\nabla \rtil\|_{L^2(L^2(\Omega))}^2
+\|\partial_\nu \rtil\|_{L^2(L^2(\partial\Omega))}^2
\]
or 
\[\lesssim 
\mathcal{E}_{me}(u)
+\taubar \|\nabla \rtil_t\|_{L^2(L^2(\Omega))}^2+\|\partial_\nu \rtil\|_{L^2(L^2(\Gamma))}^2 +\|\nabla \rtil\|_{L^2(L^2(\Omega))}^2+\|\partial_\nu \rtil\|_{L^2(L^2(\Gamma))}^2,
\]
where we hope to use the absorbing/impedance boundary conditions on $u$ in order to represent $\partial_\nu \rtil$; however, this is only possible for the $r^t$ part.
}
\end{itemize}

\subsubsection*{Full norm estimates} 
Concluding full $H^2(\Omega)$ regularity of $u$ or $\nabla u$ from $L^2(\Omega)$ boundedness of $\Delta u$ or $\Delta \nabla u$ by means of elliptic regularity is impeded by the mixed boundary conditions.
We will therefore use Stampacchia's method on the medium energy level and restrict attention to simple boundary configurations on the high energy level.

On the medium energy level, we have $\|u(t)\|_{L^\infty(\Omega)}\lesssim\|u(t)\|_{H^1(\Omega)}$ in case $d=1$ and can make use of Stampacchia's method to obtain an $L^\infty(\Omega)$ bound in case $d\in\{2,3\}$; cf., e.g. \cite[Proposition 4.1]{Consiglieri:2014}, which for a Lipschitz domain $\Omega$ yields 
\def\rrr{\mathfrak{r}}
\def\sss{\mathfrak{s}}
\[
\|u(t)\|_{L^\infty(\Omega)}\leq K(\rrr,\sss) (\| -\Delta u(t)\|_{L^{\rrr}(\Omega)} +\| \text{tr}_{\Gamma_D}u(t)\|_{L^\infty(\Gamma_D)}+\| \partial_\nu u(t)\|_{L^{\sss}(\partial\Omega\setminus\Gamma_D)})
\]
for any $\rrr>d/2$, $\sss>d-1$, $d\in\{2,3\}$. Using this with homogeneous Dirichlet, Neumann, impedance and absorbing boundary conditions yields
\[
\|u(t)\|_{L^\infty(\Omega)}\leq K(\rrr,\sss) (\| -\Delta u(t)\|_{L^{\rrr}(\Omega)} 
+\| \beta u_t(t)+\gamma u(t)\|_{L^{\sss}(\Gamma_a)})
+\| \gamma u(t)\|_{L^{\sss}(\Gamma_i)})
\]
hence with $\rrr=2$, $\sss\in(d-1,\frac{2(d-1)}{d-2})$
\[
\begin{aligned}
\|u\|_{H^1(L^\infty(\Omega))}^2
&\leq K(2,\sss)^2 (
\mathcal{E}_{me}(u) 
+\|\beta\|_{L^\infty(\Gamma_a)}\|\text{tr}_{\Gamma_a}\|_{H^1(\Omega)\to L^\sss(\Gamma_a)} 
\tfrac{1}{\taubar}\mathcal{E}_{me}(u)
\\
&\hspace*{2cm}+\|\gamma\|_{L^\infty(\Gamma)}\|\text{tr}_{\Gamma}\|_{H^1(\Omega)\to L^\sss(\Gamma)} \mathcal{E}_{lo}(u)
)\\
&\leq C_\infty 
(1+\frac{1}{\taubar})\bar{\mathcal{E}}_{me}(u).
\end{aligned}
\]
Moreover, in case $\Gamma_D=\emptyset$, an $H^{3/2}(\Omega)$ estimate can be concluded as follows
\[
\begin{aligned}
\|u \|_{H^1(H^{3/2}(\Omega))}^2
&\leq C_{ell,\frac12}\Bigl(\|\Delta u \|_{H^1(L^2(\Omega))}^2 + \|\partial_\nu u\|_{H^1(L^2(\partial\Omega))}^2\Bigr)\\
&\leq C_{ell,\frac12}\Bigl(\bar{\mathcal{E}}_{me}(u) 
+ \|\beta\|_{L^\infty(\Gamma_a)}\|u_t\|_{H^1(L^2(\Gamma_a))}^2
+\|\sqrt{\gamma} u\|_{H^1(L^2(\Gamma))}^2
\Bigr)
\\
&\leq C_{ell,\frac12}\Bigl(\bar{\mathcal{E}}_{me}(u) 
+ \max\{\tfrac{1}{\taubar},\,\|\text{tr}_{\Gamma_a}\|_{H^1(\Omega)\to L^2(\Gamma_a)},1\}\mathcal{E}_{lo}(u)\Bigr)\\
&\leq C_0 
(1+\frac{1}{\taubar})\bar{\mathcal{E}}_{me}(u).
\end{aligned}
\]

On the high energy level, we can obtain a full $H^1(0,T;H^3(\Omega))$ norm (and therewith an $L^\infty(0,T;W^{1,\infty}(\Omega))$) estimate of $u$ in the pure impedance boundary case $\partial\Omega=\Gamma_i\in C^{1,1}$ as follows.
\begin{equation}\label{a}
\begin{aligned}
&\|\partial_\nu u\|_{H^1(H^{1/2}(\Gamma_i))}^2 
=\|-\gamma \text{tr}_{\Gamma_i}u \|_{H^1(H^{1/2}(\Gamma_i))}^2
\leq C_{KP}^2 \|\gamma\|_{W^{1/2,\infty}(\Gamma_i)}^2 \|\text{tr}_{\Gamma_i}u \|_{H^1(H^{1/2}(\Gamma_i))}^2\\
&\leq C_{KP}^2 \|\gamma\|_{W^{1/2,\infty}(\Gamma_i)}^2
\|\text{tr}_{\Gamma_i}\|_{H^1(\Omega)\to H^{1/2}(\Gamma_i)}^2\, \|u \|_{H^1(H^1(\Omega))}^2\\
&\leq C_{KP}^2 \|\gamma\|_{W^{1/2,\infty}(\Gamma_i)}^2 \|\text{tr}_{\Gamma_i}\|_{H^1(\Omega)\to H^{1/2}(\Gamma_i)}^2\, \mathcal{E}_{lo}(u),
\end{aligned}
\end{equation}
where we have used a fractional Leibnitz rule estimate (often called Kato-Ponce inequality) in the first estimate, and further
\[
\begin{aligned}
\|u \|_{H^1(H^2(\Omega))}^2
&\leq C_{ell,1}\Bigl(\|\Delta u \|_{H^1(L^2(\Omega))}^2 + \|\partial_\nu u\|_{H^1(H^{1/2}(\partial\Omega))}^2\Bigr)\\
&\leq C_1 \bar{\mathcal{E}}_{me}(u).
\end{aligned}
\]
Repeating this procedure at one smoothess level higher (where the Kato Ponce estimate works slightly differently, since $H^{3/2}(\Gamma_i)$ is a Banach algebra), we get
\begin{equation}\label{b}
\begin{aligned}
&\|\partial_\nu u\|_{H^1(H^{3/2}(\Gamma_i))}^2 
=\|-\gamma \text{tr}_{\Gamma_i}u \|_{H^1(H^{3/2}(\Gamma_i))}^2
\leq C_{KP}^2 \|\gamma\|_{H^{3/2}(\Gamma_i)}^2 \|\text{tr}_{\Gamma_i}u \|_{H^1(H^{3/2}(\Gamma_i))}^2
\\
&\leq C_{KP}^2 \|\gamma\|_{H^{3/2}(\Gamma_i)}^2
\|\text{tr}_{\Gamma_i}\|_{H^2(\Omega)\to H^{3/2}(\Gamma_i)}^2\, \|u \|_{H^1(H^2(\Omega))}^2\\
&\leq C_{KP}^2 \|\gamma\|_{H^{3/2}(\Gamma_i)}^2 \|\text{tr}_{\Gamma_i}\|_{H^2(\Omega)\to H^{3/2}(\Gamma_i)}^2\, C_1 \bar{\mathcal{E}}_{me}(u)
\end{aligned}
\end{equation}
which together with 
\[
\begin{aligned}
&\|\Delta u \|_{H^1(H^1(\Omega))}^2
\leq C_{1\,PF} (\|\nabla\Delta u \|_{H^1(L^2(\Omega))}^2 + \|
\Delta u \|_{H^1(L^2(\Gamma_i))}^2)
\leq \tilde{C} \bar{\mathcal{E}}_{hi}(u).
\end{aligned}
\]
we can insert into the elliptic regularity estimate
\[
\begin{aligned}
\|u \|_{H^1(H^3(\Omega))}^2
&\leq C_{ell,2}\Bigl(\|\Delta u \|_{H^1(H^1(\Omega))}^2 + \|\partial_\nu u\|_{H^1(H^{3/2}(\partial\Omega))}^2\Bigr)\\
&\leq C_2 \bar{\mathcal{E}}_{hi}(u).
\end{aligned}
\]
This also works in the pure Neumann $\partial\Omega=\Gamma_i$ and in the pure Dirichlet case $\partial\Omega=\Gamma_D$, but not with mixed boundary conditions due to lack of regularity, nor with pure absorbing conditions $\partial\Omega=\Gamma_a$, since this would lead to higher order in time boundary terms in steps \eqref{a}, \eqref{b} of the derivation above. 

Note that we did not use the $\taubar$ terms of the respective energies in these full norm estimates and kept track of $\taubar$ dependence.
Therefore the constants $C_\infty$, $C_0$, $C_1$, $C_2$ are independent of $\tau$ and $\taubar$. 

\subsubsection*{Weak limits 
\commBK{and higher temporal regularity}
}
Energy and full norm estimates imply existence of weakly 
\commBK{\footnote{no weak* needed since all spaces are Hilbert}}
(in the weak topology of $U_{lo}$ induced by $\mathcal{E}_{lo}$) convergent subsequences with limits $\bar{u}^{lo}_M$ of $(u_{MJ})_{J\in\mathbb{N}}$ for each $M\in\mathbb{N}$ as well as $\bar{\bar{u}}^{lo}$ of $(\bar{u}^{lo}_M)_{M\in\mathbb{N}}$. 
Likewise, under the additional conditions of (i) and (ii), we have existience of subsequential limits (in the respect weak topologies induced by $\bar{\mathcal{E}}_{me}$ and $\bar{\mathcal{E}}_{hi}$, respectively) $\bar{\bar{u}}^{me}$ and $\bar{\bar{u}}^{hi}$, respectively, that due to uniqueness of limits have to coincide 
$\bar{\bar{u}}^{lo}=\bar{\bar{u}}^{me}=\bar{\bar{u}}^{hi}=:\bar{\bar{u}}$.

To prove that $\bar{\bar{u}}$ solves \eqref{ve_lo}, we proceed as in the usual way, see, e.g., \cite{EvansBook}, with the slight modification of considering a discrete - continuous transition not only on the (inner) spatial but also on the (outer) temporal level.
For each $M, \, j\in\mathbb{N}$ and 
for every $J\geq j$, we have that $u_{MJ}$ satisfies $(ve)$ for all $v\in U_{Mj}$; 
taking the limit as $J\to\infty$, we conclude that also $\bar{u}^{lo}_M$ satisfies $(ve)$ for all $v\in U_{Mj}$;
since $j\in\mathbb{N}$ was arbitrary, $\bar{u}^{lo}_M$ satisfies $(ve)$ for all $v\in \bar{U}_M:= \overline{\bigcup_{j\in\mathbb{N}}U_{Mj}}=
\{(t,x)\mapsto\Re\Bigl(\sum_{m=1}^M\exp(\imath m\omega t)\hat{v}_m(x)\Bigr)\, : \, \hat{v}_m\in H^1_D(\Omega)\}$.
Now for any $m\in\mathbb{N}$ and 
for every $M\geq m$, we have that $\bar{u}^{lo}_M$ satisfies $(ve)$ for all $v\in \bar{U}_m$;
taking the limit as $M\to\infty$, we conclude that also $\bar{\bar{u}}^{lo}$ satisfies $(ve)$ for all $v\in \bar{U}_m$;
since $m\in\mathbb{N}$ was arbitrary, $\bar{\bar{u}}^{lo}$ satisfies $(ve)$ for all $v\in V_{lo}= \overline{\bigcup_{m\in\mathbb{N}}\bar{U}_m}$;
Thus, we have constructed a variational solution to \eqref{JMGT_linearized_rtil} with \eqref{periodic}, \eqref{bndy}.

\subsubsection*{Uniqueness}
Since the testing that led to the low regularity energy estimate can as well be carried out on a continuos level, the homogeneous equation with $\tilde{r}=0$ only has the trivial solution, which due to linearity implies uniqueness.

\medskip

Altogether we have proven Theorem~\ref{thm:JMGT-linearized}. 

\begin{remark}\label{rem:Eme}
In order to recover $\|u\|_{L^2(L^2(\Omega))}$, instead of combining \eqref{enid_me} with \eqref{enid_lo} in \eqref{energies_me}, we can make use of \eqref{Cell0} with \eqref{enid_me} alone 
in case $\Gamma_a=\emptyset$.
This allows us to avoid the differentiability assumption \eqref{smallnesscoeffs_lo} on the coefficients $b$ and $c^2$ in this case.
\end{remark}

\section{Nonlinear well-posedness and differentiability: proof of Theorems~\ref{thm:JMGT-Westervelt}, \ref{thm:JMGT-Kuznetsov}, 
\skipKuznetsov{\ref{thm:Kuznetsov},} 
\ref{thm:diff}}\label{sec:nonlinearwellposedness}

Given the linear results from Theorem~\ref{thm:JMGT-linearized}, the well-posendess proof for the nonlinear equations \eqref{JMGT-Westervelt}, \eqref{JMGT-Kuznetsov} under periodicity conditions proceeds along the lines of the initial value problem analogs, in, e.g., \cite{KLM12_MooreGibson,KLP12_JordanMooreGibson,JMGT,JMGT_Neumann}.
We therefore here only elaborate the essential arguments.

To prove existence and uniqueness of solutions to
\begin{equation}\label{PDE_DN}
\begin{aligned}
&\mathcal{D} u + \mathcal{N}(u)=0\ \text{ with \eqref{periodic} and \eqref{bndy}}\\
&\text{ where }\begin{cases}
\mathcal{D}=\tau \partial_t^3 + \partial_t^2 + (b\partial_t+c^2\text{id})(-\Delta)\\
\mathcal{N}(u)=\eta (u^2)_{tt} +f\text{ or }\mathcal{N}(u)=\bigl(\tilde{\eta}u_t^2+|\nabla u|^2\bigr)_t +f
\end{cases}
\end{aligned}
\end{equation}
we apply Banach's contraction principle to the operator $\mathcal{T}:\mathcal{B}_\radius^U(\subseteq U)\to U$, $u^-\mapsto u$ solving \eqref{PDE_DN} with $\mathcal{N}(u)$ replaced by $\mathcal{N}(u^-)$.
Here $\mathcal{B}_\radius^U$ is a sufficiently small ball and $U=U_{me}$ for JMGT-Westervelt, $U=U_{hi}$ for JMGT-Kuznetsov, $f\in\mathcal{B}_{\radius_{J*}}^Z(\subseteq Z)$ with $\radius_{J*}>0$ chosen small enough and the two corresponding cases $Z=Z_{me}$, $Z=Z_{hi}$.

Well-definedness and the self-mapping property of $\mathcal{T}$ follows from parts (ii) and (iii) of Theorem~\ref{thm:JMGT-linearized}, as long as we can prove that the coefficients satisfy the conditions of this theorem and $\mathcal{N}(u^-)$ is small enough in $Z$.
In fact we will obtain slightly more than invariance on $\mathcal{B}_\radius^U$, namely $\mathcal{T}(\mathcal{B}_\radius^{U^0})\subseteq \mathcal{B}_\radius^U$ with $U^0$ as defined in \eqref{U0}; that is, the images under $\mathcal{T}$ will be estimated without using the $\tau$ dependent parts in the energy. This allows us to include the case $\tau=0$ in Theorems~\ref{thm:JMGT-Westervelt}, \ref{thm:JMGT-Kuznetsov}.

The contraction property of $\mathcal{T}$ is achieved by considering the (linear) PDE satisfied by the difference between two solutions $\hat{u}=u_1-u_2$ where $u_i=T(u^-_i)$, $u^-_i\in \mathcal{B}_\radius^U$, $i=1,2$ and $\hat{u}^-=u^-_1-u^-_2$. Again, Theorem~\ref{thm:JMGT-linearized} is essential for obtaining a bound $\|\hat{u}\|_U\leq L\|\hat{u}^-\|_U$ and by smallness of $\mathcal{B}_\radius^U$ we can achieve that the Lipschitz constant $L$ is smaller than unity.

Since very similar bounds and again an application of Theorem~\ref{thm:JMGT-linearized} lead to differentiability and Lipschitz continuity of the derivative of the source-to-state map $\mathcal{S}:Z\to U$, $f\mapsto u$ solving \eqref{PDE_DN}, we will carry out the estimates for all these purposes jointly. 

With the notation from \eqref{PDE_DN},
the tasks
\[
\begin{aligned}
&\text{(sm) 
$\mathcal{T}(\mathcal{B}_\radius^{U^0})\subseteq \mathcal{B}_\radius^U$,} \quad&& 
\text{(df0) existence of a 
first variation
$\mathcal{S}'(u)$,} \\
&\text{(ct) contractivity of $\mathcal{T}$,} && 
\text{(df1) Fr\'{e}chet differentiability of $\mathcal{S}$,}\\
&&&\text{(df2) Lipschitz continuity of the derivative $\mathcal{S}'$,}
\end{aligned}
\]
for $\hat{u}=u_1-u_2$, $\hat{u}^-=u^-_1-u_2^-$, $\uld{u}=\mathcal{S}'(f)\uld{f}$, $\tilde{u}=\mathcal{S}(f+\uld{f})-\mathcal{S}(f)-\mathcal{S}'(f)\uld{f}$, $\tilde{\tilde{u}}=\mathcal{S}'(f_\sim)\uld{f}-\mathcal{S}'(f)\uld{f}$, $u_\sim=\mathcal{S}(f_\sim)$
read as
\[
\begin{aligned}
\text{(sm) }&\mathcal{D} u + \mathcal{N}(u^-)=0 \ \stackrel{!}{\Rightarrow} \ \|u\|_U\leq C (\|u\|_U^2+\|f\|_Z)\leq C(\radius^2+\radius_{J*})
\\
\text{(ct) }&\mathcal{D} \hat{u} + \mathcal{N}(u^-_1)-\mathcal{N}(u^-_2)=0 \ \stackrel{!}{\Rightarrow} \ \|\hat{u}\|_U\leq C (\|u^-_1\|_U+\|u^-_2\|_U) \|\hat{u}^-\|_U\leq 2C\radius \|\hat{u}^-\|_U\\
\text{(df0) }&\mathcal{D} \uld{u} + \mathcal{N}'(u)\uld{u}+\uld{f}=0 \ \stackrel{!}{\Rightarrow} \ \|\uld{u}\|_U\leq C (\|u\|_U\|\uld{u}\|_U+\|\uld{f}\|_Z) \\
&\hspace*{4.9cm} \leq C (\radius \|\uld{u}\|_U+\|\uld{f}\|_Z) 
\ \Rightarrow \ \|\uld{u}\|_U\leq \tfrac{C}{1-C\radius}\|\uld{f}\|_Z
\\
\text{(df1) }&\mathcal{D} \tilde{u} + \mathcal{N}(u+\uld{u})-\mathcal{N}(u)-\mathcal{N}'(u)\uld{u}=0 \ \stackrel{!}{\Rightarrow} \ \|\tilde{u}\|_U\leq C(\|u\|_U+1) \|\uld{u}\|_U^2 \\
\text{(df2) }&\mathcal{D} \tilde{\tilde{u}} + (\mathcal{N}'(u_\sim)-\mathcal{N}'(u))\uld{u}=0 \ \stackrel{!}{\Rightarrow} \ \|\tilde{\tilde{u}}\|_U\leq C (\|u_\sim\|_U+\|u\|_U)\|u_\sim-u\|_U\|\uld{u}\|_U\\
&\hspace*{8.3cm}\leq 2C\radius\left(\tfrac{C}{1-C\radius}\right)^2\|f_\sim-f\|_U\|\uld{f}\|_U
\end{aligned}
\]
with
\[
C(\radius^2+\radius_{J*})\leq\radius\quad\text{ and }\quad 2C\radius<1
\]
provided $\radius$ and $\radius_{J*}$ are chosen small enough.

The proof step indicated by $\stackrel{!}{\Rightarrow}$ consists of a combination of the application of Theorem~\ref{thm:JMGT-linearized} with a verification of its assumptions for the inhomogeneities.
The following choices of $\rtil$ will thus be relevant for these tasks.
\begin{itemize}
\item In the JMGT-Westervelt case \eqref{JMGT-Westervelt}:
\\
The choice directly corresponding to \eqref{alphar} and analogous to \cite{periodicWestervelt} would be 
\\[1ex]
\begin{tabular}{l||l|l}
&$\alpha$		&$\rtil$		
\\ \hline
(sm)& $1+2\eta u^-$		&$2\eta u^-_t\, u_t +f$\\
(ct)& $1+2\eta u^-_1$	&$2\eta (\hat{u}^-u_{2tt}+u^-_{1t}\, u_{1t}-u^-_{2t}\, u_{2t})$\\
(df0)& $1+2\eta u$	&$
4\eta u_t\uld{u}_t+2\eta u_{tt}\, \uld{u}
+\uld{f}$
\end{tabular}
\\[1ex]
We here follow the choice $\alpha=1$ analogous to \cite{periodicWest_2}, though, 
which simplifies the treatment of the terms pertaining to derivatives $\alpha$ in \eqref{smallnesscoeffs_lo}, \eqref{smallnesscoeffs_me} (at the cost of having to estimate further terms in $\rtil$; 
note that the need of $L^\infty(0,T;L^\infty(\Omega))$ estimates of the states $u^-$, $u^-_1$, $u$, persists.)
\\[1ex]
\begin{tabular}{l|l}
&$\rtil$		
\\ \hline
(sm)&$\eta \bigl((u^-)^2\bigr)_{tt}+f
$\\
(ct)&$\eta \bigl((u^-_1)^2-(u^-_2)^2\bigr)_{tt}
				 =\eta \bigl((u^-_1+u^-_2)\hat{u}^-\bigr)_{tt}
$\\
(df0)&$2\eta \bigl(u\uld{u}\bigr)_{tt}
+\uld{f}
$\\
(df1)&$2\eta \bigl((\uld{u})^2\bigr)_{tt}
$\\
(df2)&$2\eta \bigl((u_\sim-u)\uld{u}\bigr)_{tt}
$
\end{tabular}
\\[1ex]
Our goal is to bound $\rtil$ on the medium level energy 
\begin{equation}\label{todos_me}
\begin{aligned}
&\|\rtil_{(sm)}\|_{Z_{me}}\lesssim
\|u^-\|_{U_{me}}^2
+\|f\|_{Z_{me}} \\
&\|\rtil_{(ct)}\|_{Z_{me}}\lesssim
(\|u^-_1\|_{U_{me}}+\|u^-_2\|_{U_{me}})\, \|\hat{u}^-\|_{U_{me}}
\\
&\|\rtil_{(df0)}\|_{Z_{me}}\lesssim
\|u\|_{U_{me}}\, \|\uld{u}\|_{U_{me}}
+\|\uld{f}\|_{Z_{me}}\\
&\|\rtil_{(df1)}\|_{Z_{me}}\lesssim
(\|u\|_{U_{me}}+1)\|\uld{u}\|_{U_{me}}^2\\
&\|\rtil_{(df2)}\|_{Z_{me}}\lesssim
(\|u_\sim\|_{U_{me}}+\|u\|_{U_{me}})\|u_\sim-u\|_{U_{me}}\, \|\uld{u}\|_{U_{me}}
.
\end{aligned}
\end{equation}
To this end, we use the fact that they are of similar structure, namely, up to the $f$ and $\uld{f}$ terms, they are of the form
\[
r[v,w]=\eta(v\,w)_{tt}=\eta(v\,w_{tt}+2v_tw_t+v_{tt}w) 
\]
with 
\begin{equation}\label{vw}
\begin{cases}
(sm)&\ v=w=u^-\\
(ct)&\ v=u^-_1+u^-_2, \ w=\hat{u}^-\\
(df0)&\ v=2u, \ w=\uld{u}\\
(df1)&\ v=w=2\uld{u}\\
(df2)&\ v=2(u_\sim-u), \ w=\uld{u}\\
\end{cases} 
\end{equation}
and $r[v,w]$ inherits the homogeneous Dirichlet boundary conditions on $\Gamma_D$ from the states $u^-$, $u^-_i$, $u$, and $u_\sim$, respectively, so that we can set $\rtil^t=r[v,w]$, $r^\nabla=f$.
It thus suffices to estimate
\[
\begin{aligned}
&\|\nabla r[v,w]\|_{L^2(L^2(\Omega))}=\|\nabla \eta (v\,w_{tt}+2v_tw_t+v_{tt}w)\\
&\qquad+\eta (\nabla v\,w_{tt}+v\,\nabla w_{tt}+2\nabla v_t\,w_t+2v_t\,\nabla w_t+\nabla v_{tt}\, w+v_{tt}\, \nabla w)\|_{L^2(L^2(\Omega))}
\end{aligned}
\]
where 
\[
\begin{aligned}
&\|\nabla \eta \, v\,w_{tt} \|_{L^2(L^2(\Omega))}
\leq \|\nabla \eta\|_{L^{\hat{p}}(\Omega)} \|v\|_{L^\infty(L^\infty(\Omega))} \|w_{tt} \|_{L^2(L^p(\Omega))}
\\
&\|\nabla \eta \, v_t\, w_t \|_{L^2(L^2(\Omega))}
\leq \|\nabla \eta\|_{L^{\hat{p}}(\Omega)} \|v_t\|_{L^4(L^{2p}(\Omega))} \|w_t \|_{L^4(L^{2p}(\Omega))}
\\
&\|\eta\, \nabla v\,w_{tt}\|_{L^2(L^2(\Omega))}
\leq \|\eta\|_{L^\infty(\Omega)} \|\nabla v\|_{L^\infty(L^{\hat{p}}(\Omega))} \|w_{tt} \|_{L^2(L^p(\Omega))}
\\
&\|\eta\, v\,\nabla w_{tt}\|_{L^2(L^2(\Omega))}
\leq \|\eta\|_{L^\infty(\Omega)} \|v\|_{L^\infty(L^\infty(\Omega))} \|\nabla w_{tt} \|_{L^2(L^2(\Omega))}
\\
&\|\eta\, \nabla v_t\,w_t\|_{L^2(L^2(\Omega))}
\leq \|\eta\|_{L^\infty(\Omega)} \|\nabla v_t\|_{L^\infty(L^{\hat{q}}(\Omega))} \|w_t \|_{L^2(L^q(\Omega))}
\end{aligned}
\]
and by interpolation
\[
\begin{aligned}
\|\nabla v_t\|_{L^\infty(L^{\hat{q}}(\Omega))}
&\leq C_{H^{3/4}\to L^\infty}^{(0,T)} C_{H^{1/8}\to L^{\hat{q}}}^\Omega
\|\nabla v_t\|_{H^{3/4}(H^{1/8}(\Omega))}\\
&\leq C_{H^{3/4}\to L^\infty}^{(0,T)} C_{H^{1/8}\to L^{\hat{q}}}^\Omega
\|\nabla v_t\|_{H^1(L^2(\Omega))}^{3/4}
\|\nabla v_t\|_{L^2(H^{1/2}(\Omega))}^{1/4}.
\end{aligned}
\]
Here we choose $p=3$, $\hat{p}=6$, $2p=6$, $q=8d$, $\hat{q}=\frac{8d}{4d-1}$, so that the embeddings $H^1(\Omega)\to L^{\hat{p}}$, $H^1(\Omega)\to L^{2p}$, $H^{3/2}(\Omega)\to L^{q}$, $H^{1/8}(\Omega)\to L^{\hat{q}}$ are continuous.\\
The remaining terms can be estimated analogously.
Altogether, this implies \eqref{todos_me}.
\item In the JMGT-Kuznetsov case \eqref{JMGT-Kuznetsov}:
\\[1ex]
\begin{tabular}{l|l}
&$\rtil$		
\\ \hline
(sm)&$\bigl(\eta(u^-_t)^2+|\nabla u^-|^2\bigr)_t+f$\\
(ct)&$\bigl(\eta(u^-_1+u^-_2)_t\hat{u}^-_t+\nabla (u^-_1+u^-_2)\cdot\nabla\hat{u}^-\bigr)_t$\\
(df0)&$2\bigl(\eta u_t\uld{u}_t+\nabla u\cdot\nabla\uld{u}\bigr)_t +\uld{f}$\\
(df1)&$ 2\bigl(\eta \uld{u}_t^2+|\nabla\uld{u}|^2\bigr)_t$\\
(df2)&$2\bigl(\eta (u_\sim-u)_t\uld{u}_t+\nabla (u_\sim-u)\cdot\nabla\uld{u}\bigr)_t$
\end{tabular}
\\[1ex]
and again $\alpha=1$.
We will work on the high order energy level, thus replacing $me$ by $hi$ in \eqref{todos_me}, again using some  common bilinear structure, which here reads as 
\begin{equation}\label{r_K}
r[v,w]=\bigl(\eta v_t\,w_t+\nabla v\cdot\nabla w\bigr)_t=\eta(v_t\,w_{tt}+v_{tt}\,w_t)+\nabla v\cdot\nabla w_t+\nabla v_t\cdot\nabla w
\end{equation}
with $v,w$ as in \eqref{vw}.
To this end, we have to bound
\[
\begin{aligned}
&\|\Delta r[v,w]\|_{L^2(L^2(\Omega))}\\
&=\|\Delta\eta(v_t\,w_{tt}+v_{tt}w_t)+2\nabla \eta\cdot(\nabla v_t\,w_{tt}+v_t\,\nabla w_{tt}+\nabla v_{tt}\,w_t+v_{tt}\, \nabla w_t)\\
&\qquad+\eta(\Delta v_t\,w_{tt}+2\nabla v_t\cdot \nabla w_{tt}+v_t\,\Delta w_{tt}+\Delta v_{tt}\,w_t+2\nabla v_{tt}\cdot\nabla w_t+v_{tt}\,\Delta w_t)\\
&\qquad+ \nabla \Delta v\cdot\nabla w_t+2D^2 v:D^2 w_t+\nabla v\cdot\nabla \Delta w_t\\
&\qquad+\nabla\Delta  v_t\cdot\nabla w+2D^2 v_t:D^2 w+\nabla v_t\cdot\nabla \Delta w \|_{L^2(L^2(\Omega))}
\end{aligned}
\]
where
\[
\begin{aligned}
&\|\Delta\eta\,v_t\,w_{tt} \|_{L^2(L^2(\Omega))}
\leq \|\Delta\eta\|_{L^2(\Omega)} \|v_t\|_{L^\infty(L^\infty(\Omega))} \|w_{tt} \|_{L^2(L^\infty(\Omega))}
\\
&\|\nabla \eta\cdot\nabla v_t\,w_{tt} \|_{L^2(L^2(\Omega))}
\leq \|\nabla\eta\|_{L^{\hat{p}}(\Omega)} \|\nabla v_t\|_{L^\infty(L^p(\Omega))} \|w_{tt} \|_{L^2(L^\infty(\Omega))}
\\
&\|\nabla \eta\cdot v_t\,\nabla w_{tt}\|_{L^2(L^2(\Omega))}
\leq \|\nabla\eta\|_{L^{\hat{p}}(\Omega)} \| v_t\|_{L^\infty(L^\infty(\Omega))} \|\nabla w_{tt} \|_{L^2(L^p(\Omega))}
\\
&\|\eta\,\Delta v_t\,w_{tt} \|_{L^2(L^2(\Omega))}
\leq\|\eta\|_{L^\infty(\Omega)} \|\Delta v_t \|_{L^\infty(L^2(\Omega))} \|w_{tt} \|_{L^2(L^\infty(\Omega))}
\\
&\|\eta\,\nabla v_t\cdot \nabla w_{tt} \|_{L^2(L^2(\Omega))}
\leq\|\eta\|_{L^\infty(\Omega)} \|\nabla v_t\|_{L^\infty(L^{\hat{p}}(\Omega))} \|\nabla w_{tt} \|_{L^2(L^p(\Omega))}
\\
&\|\eta\,v_t\,\Delta w_{tt} \|_{L^2(L^2(\Omega))}
\leq\|\eta\|_{L^\infty(\Omega)} \| v_t\|_{L^\infty(L^\infty(\Omega))} \|\Delta w_{tt} \|_{L^2(L^2(\Omega))}
\\
&\|\nabla \Delta v\cdot\nabla w_t \|_{L^2(L^2(\Omega))}
\leq \|\nabla \Delta v\|_{L^\infty(L^2(\Omega))}\|\nabla w_t \|_{L^2(L^\infty(\Omega))}
\\
&\|D^2 v:D^2 w_t \|_{L^2(L^2(\Omega))}
\leq \|D^2 v\|_{L^\infty(L^4(\Omega))}\|D^2 w_t \|_{L^2(L^4(\Omega))}
\\
&\|\nabla v\cdot\nabla \Delta w_t \|_{L^2(L^2(\Omega))}
\leq \|\nabla v\|_{L^\infty(L^\infty(\Omega))} \|\nabla \Delta w_t \|_{L^2(L^2(\Omega))},
\end{aligned}
\]
and the rest can again be estimated analogously.

\end{itemize}

\bigskip

\commBK{
For (Lipschitz continuous) differentiability of the source-to-state map $S:f\mapsto u$ we make use of the bilinear structure of the nonlinearity:\\ 
A bilinear operator 
\[
B:O^2\subseteq U^2\to Z, \qquad (u,v)\mapsto B(u,v)
\] 
with $O=\mathcal{U}_\radius(0)$ being a neighborhood of zero has Lipschitz continuous Fr\'{e}chet derivative 
$B'(u,v)(\uld{u},\uld{v})=B(\uld{u},v)+B(u,\uld{v})$ with Taylor remainder 
$B(u+\uld{u},v+\uld{v})-B(u,v)-B'(u,v)(\uld{u},\uld{v})=B(\uld{u},\uld{v})$, if $B$ is bounded in the sense that
\begin{equation}\label{Bbounded}
\|B(u,v)\|_Z\leq C \|u\|_U \|v\|_U \text{ for all }u,v\in O
\end{equation}
holds.
We combine this with an implicit function theorem argument using the fact that $S$ is implicitly defined by 
\[
0=A(u,f)=Lu+B(u,u)+f \quad \Leftrightarrow \quad u=S(f)
\]
where $L:U\to Z$ is the differential operator 
\[
L=\tau\partial_t^3+\partial_t^2+(b\partial_t+c^2)(-\Delta)
\]
Indeed the assumptions of the implicit function theorem are satisfied as follows.
\begin{itemize}
\item $A(0,0)=0$
\item $\partial_u A(u,f)=L+B(u,\cdot)+B(\cdot,u):U\to Z$ is an isomorphism for the function space pairs $(U,Z)=(U_{me},Z_{me})$ and $(U,Z)=(U_{hi},Z_{hi})$ for the JMGT-Westervelt and JMGT-Kuznetsov case, respectively.
\item $\partial_u A$ is Lipschitz continuous due to \eqref{Bbounded}.
\end{itemize}
Thus $S$ is well-defined and Lipschitz continuously differentiable in a neighborhood of zero.\\
This proves Theorem~\ref{thm:diff}.
}

\medskip

\commBK{
next paper:
\begin{itemize}
\item multiharmonic expansion:\\
maybe do simulations with Benjamin's code, using the fact that the left hand side in the multiharmonic system is still a Helmholtz differential operator 
$-\Delta-\kappa_m^2$ with 
$\kappa_m^2=\frac{\imath m^3\omega^3+m^2\omega^2}{\imath m\omega b+c^2}$
\item inverse problem:\\ 
Fr\'{e}chet differentiability of red (and aao) forward operator (setting $\uld{f}=\uld{\eta}(u^2)_tt$;\\ linearized uniqueness
(as in nonlinearityimaging-2d, but using unique continuation, see ECM2024 Invited; but mention that $M$ is a compact perturbation of the identity (multiply both sides with $m\omega$ if needed); re-define $t_m=\frac{2\pi}{m\omega}$ to give it a physical meaning: length of minimal period in time domain)\\
\end{itemize}
}

\section*{Acknowledgment}
This research was funded in part by the Austrian Science Fund (FWF) 
[10.55776/P36318]. 

\end{document}